\documentclass[12pt,reqno,oneside]{amsart}
\usepackage{geometry}                
\usepackage{graphicx}
\usepackage{amssymb}
\usepackage{epstopdf}

\usepackage{geometry}
\usepackage{graphicx}

\usepackage{booktabs} 
\usepackage{array} 
\usepackage{paralist} 
\usepackage{verbatim} 
\usepackage{subfig} 
\usepackage{tabularx}

\usepackage{amsmath,amsfonts,amsthm,mathrsfs,amssymb,cite}
\usepackage[usenames]{color}

\newcommand{\Bx}{\mathbf{x}}
\newcommand{\By}{\mathbf{y}}

\newcommand{\p}{\partial}

\newcommand{\GG}{\Gamma}
\newcommand{\Gs}{\sigma}

\newtheorem{thm}{Theorem}[section]

\newtheorem{lem}{Lemma}[section]
\newtheorem{prop}{Proposition}[section]
\theoremstyle{definition}
\newtheorem{defn}{Definition}[section]
\theoremstyle{remark}

\newtheorem{rem}{Remark}[section]
\numberwithin{equation}{section}

\numberwithin{equation}{section}
\newcounter{saveeqn}

\newcommand{\eqnref}[1]{(\ref {#1})}

\newcommand{\Bd}{\mathbf{d}}

\newcommand{\Bn}{\mathbf{n}}

\newcommand{\Bz}{\mathbf{z}}

\newcommand{\tdx}{\tilde{\Bx}}
\newcommand{\tdy}{\tilde{\By}}

\newcommand{\Kcal}{\mathcal{K}}

\newcommand{\Scal}{\mathcal{S}}
\newcommand{\Dcal}{\mathcal{D}}

\newcommand{\Ocal}{\mathcal{O}}

\newcommand{\la}{\langle}
\newcommand{\ra}{\rangle}

\newcommand{\RR}{\mathbb{R}}

\newcommand{\beq}{\begin{equation}}
\newcommand{\eeq}{\end{equation}}

\DeclareMathAlphabet{\itbf}{OML}{cmm}{b}{it}

\title[Regularized Full- and Partial-cloaks]{On Regularized Full- and Partial-Cloaks in Acoustic Scattering}

\author{Youjun Deng}
\address{School of Mathematics and Statistics, Central South University, Changsha, Hunan, P. R. China.}
\email{youjundeng@csu.edu.cn, dengyijun\_001@163.com}

\author{Hongyu Liu}
\address{Department of Mathematics, Hong Kong Baptist University, Kowloon Tong, Hong Kong SAR.\vspace*{-4mm}}
\address{\vspace*{-4mm}and}
\address{HKBU Institute of Research and Continuing Education, Virtual University Park, Shenzhen, P. R. China.}
\email{hongyu.liuip@gmail.com}

\author{Gunther Uhlmann}
\address{Department of Mathematics, University of Washington, Seattle, WA 98195, USA.}
\email{gunther@math.washington.edu}

\begin{document}
\maketitle

\begin{abstract}
The aim of this work is to derive sharp quantitative estimates of the qualitative convergence results developed in \cite{LiLiuRonUhl} for regularized full- and partial-cloaks via the transformation-optics approach. Let $\Gamma_0$ be a compact set in $\mathbb{R}^3$ and $\Gamma_\delta$ be a $\delta$-neighborhood of $\Gamma_0$ for $\delta\in\mathbb{R}_+$. $\Gamma_\delta$ represents the virtual domain used for the blow-up construction. By incorporating suitably designed lossy layers, it is shown that if the generating set $\Gamma_0$ is a generic curve, then one would have an approximate full-cloak within $\delta^2$ to the perfect full-cloak; whereas if $\Gamma_0$ is the closure of an open subset on a flat surface, then one would have an approximate partial-cloak within $\delta$ to its perfect counterpart. The estimates derived are independent of the contents being cloaked; that is, the cloaking devices are capable of nearly cloaking an arbitrary content. Furthermore, as a significant byproduct, our argument allows the relaxation of the convexity requirement on $\Gamma_0$ in \cite{LiLiuRonUhl}, which is critical for the Mosco convergence argument therein.

\medskip

\medskip



\end{abstract}

\section{Introduction}

This work concerns the invisibility cloaking for time-harmonic scalar waves governed by the Helmholtz system via the approach of transformation optics. A proposal for cloaking for electrostatics using the invariance properties
of the conductivity equation was pioneered in \cite{GLU,GLU2}.
Blueprints for making objects invisible to electromagnetic (EM) waves were proposed in two articles in {\it Science} in 2006 \cite{Leo,PenSchSmi}. The article by Pendry et al uses the same transformation used in \cite{GLU,GLU2} while the work of Leonhardt uses a conformal mapping in two dimensions. The method based on the invariance properties of the equations modelling the wave phenomenon has been named
{\sl transformation optics}. There have been several other proposals for cloaking \cite{AE,MN}. The transformation optics has received a lot of attention in the scientific community and the popular press because of the generality of the method; see \cite{CC,GKLU4,GKLU5,LiuUhl,U2} for comprehensive surveys on the theoretical and experimental progress.

Using the transformation optics, the construction of an ideal cloak is based on blowing up a point in the virtual space which splits a `hole' in the physical space to form the cloaked region, whereas the ambient space is `compressed' via the {\it push-forward} to form the cloaking region. In a similar fashion cloaking devices based on blowing up a {\it crack} (namely, a curve in $\mathbb{R}^3$) or a {\it screen} (namely, a flat surface in $\mathbb{R}^3$) were, respectively, considered in \cite{GKLU2,GKLU1} and \cite{LiP}, resulting in the so-called EM wormhole and carpet-cloak respectively.
The blow-up-a-point (respectively, -crack or -screen) construction yields singular cloaking materials, namely, the material parameters violate the physical regular conditions. The singular media present a great challenge for both theoretical analysis and practical fabrications. In order to avoid the singular structures, several regularized constructions have been developed in the literature. In \cite{GKLUoe,GKLU_2,RYNQ}, a truncation of singularities has been introduced. In \cite{KOVW,KSVW,Liu}, the blow-up-a-point transformation in \cite{GLU2,Leo,PenSchSmi} was regularized to become the `blow-up-a-small-region' transformation. Nevertheless, it is pointed out in \cite{KocLiuSun} that the truncation-of-singularities construction and the blow-up-a-small-region construction are equivalent to each other.  Instead of the ideal invisibility, one would consider approximate/near invisibility for a regularized construction; that is, one intends to make the corresponding wave scattering effect due to a regularized cloaking device as small as possible depending on an asymptotically small regularization parameter $\delta\in\mathbb{R}_+$.

Due to its practical importance, the regularized cloaks have been extensively investigated in the literature; see \cite{Ammari2,Ammari3,Ammari4,BL,BLZ,GKLUr,KOVW,KSVW,LiuSun,LiuZhou} and the references therein. In a recent work \cite{LiLiuRonUhl}, a mathematical framework on constructing regularized full- and partial-cloaks for time-harmonic scalar waves governed by the Helmholtz system was developed, extending the existing study in the literature to the most general case.
There are three major ingredients of the general cloaking scheme in \cite{LiLiuRonUhl}. The cloaking medium is obtained by blowing up a virtual domain $\Gamma_\delta$, which is a $\delta$-neighborhood of a compact convex generating set $\Gamma_0$. Second, a compatible lossy layer should be employed right between the cloaked and cloaking regions, which is indispensable of enabling the corresponding construction to nearly cloak an arbitrary optical contents. Third, based on some qualitative convergence analysis, it is shown that if the generating set $\Gamma_0$ has a zero capacity, then one would have an approximate full-cloak; whereas if the generating set $\Gamma_0$ has a non-zero capacity, then one would have an approximate partial-cloak. Here, by a full-cloak, we mean that the invisibility holds for incident waves coming from every possible direction and observation made in every possible angle, and otherwise the cloaking device is called a partial-cloak.

Clearly, quantitatively estimating the convergence results in \cite{LiLiuRonUhl} would yield a much more accurate characterization of the regularized cloaking constructions. Most of the available results in the literature mentioned earlier on quantitatively characterizing the degree of approximation to the ideal cloak are mainly concerned with the case that the generating set $\Gamma_0$ is a single point. For the other cases, the corresponding mathematical study in deriving sharp quantitative estimates were fraught with significant technical difficulties and were lacked in the literature. This work aims at dealing with this challenging issue. Specifically, by letting $\Gamma_0$ and $\Gamma_\delta$, $\delta\in\mathbb{R}_+$, respectively denote the generating set and virtual domain, we show that if $\Gamma_0$ is a generic curve, then one would have an approximate full-cloak within $\delta^2$ to the perfect full-cloak; whereas if $\Gamma_0$ is the closure of an open subset on a flat surface, then one would have an approximate partial-cloak within $\delta$ to its perfect counterpart. The estimates derived are independent of the contents being cloaked; that is, the cloaking devices are capable of nearly cloaking arbitrary contents. We make essential use of variational and layer potential techniques for our study. Some of the results derived in the process are of significant theoretical and practical interests for their own sake. Furthermore, as a significant byproduct, our argument allows the relaxation of the convexity requirement on $\Gamma_0$ in \cite{LiLiuRonUhl}, which is critical for the Mosco convergence argument therein.

The rest of the paper is organized as follows. In the next section, preliminaries on acoustic scattering and
transformation-optics cloaking are presented. In section 3, we introduce the layer
potentials and briefly discuss their mapping properties. Section 4 is devoted to showing the sharp quantitative estimates
of the regularized full-cloak when the generating set is a generic curve.
In section 5, sharp estimate on the regularized partial-cloak when the generating set is a flat subset. The main results in this paper are
given in Theorems \ref{th:main1} and \ref{th:main2}.

\section{Wave scattering and transformation-optics cloaking}

Let $q\in L^\infty(\mathbb{R}^3)$ be a (complex valued) scalar function and $\sigma=(\sigma^{jl})\in L^\infty(\mathbb{R}^3)^{3\times 3}$ be a real symmetric-positive-definite matrix-valued function. It is assumed that there exists $0<\lambda\leq 1$ such that
\begin{equation}\label{eq:regular}
\Re q(\mathbf{x})\geq \lambda,\ \Im q(\mathbf{x})\geq 0;\ \ \ \lambda\, \mathbf{I}_{3\times 3}\leq \sigma(\mathbf{x})\leq \lambda^{-1}\, \mathbf{I}_{3\times 3}\ \ \mbox{for a.e. $\mathbf{x}\in\mathbb{R}^3$},
\end{equation}
where $\mathbf{I}_{3\times 3}$ is the $3\times 3$ identity matrix. Physically, $q$ is the modulus and $\sigma^{-1}$ is the density tensor of an acoustic medium in $\mathbb{R}^3$. It is always assumed that the acoustic inhomogeneity is compactly supported, i.e., there is a bounded domain $\Omega\subset\mathbb{R}^3$ such that $q=1$ and $\sigma=\mathbf{I}_{3\times 3}$ in $\mathbb{R}^3\backslash\Omega$. We write $\{\Omega; \sigma, q\}$ to signify an acoustic medium as described above. \eqref{eq:regular} is the {\it regular condition} for the acoustic medium. In scattering theory, time-harmonic plane waves of the form $u^i(\mathbf{x}; \omega,\mathbf{d}):=e^{i\omega\mathbf{x}\cdot\mathbf{d}}$ impinge on the acoustic medium $\{\Omega; \sigma, q\}$ from the direction $\mathbf{d}\in\mathbb{S}^2:=\{\mathbf{x}\in\mathbb{R}^3; |\mathbf{x}|=1\}$. Here, $\omega\in\mathbb{R}_+$ denotes the wave number of the time-harmonic wave. The corresponding total wave field $u(\mathbf{x}; \omega, \mathbf{d})\in\mathbb{C}$, $\mathbf{x}:=(x_j)_{j=1}^3\in\mathbb{R}^3$, is governed by the following Helmholtz system
\begin{equation}\label{eq:Helm}
\begin{cases}
\ \displaystyle{\sum_{j,l=1}^{3}\frac{\partial}{\partial x_j}\left(\sigma^{jl}(\mathbf{x})\frac{\partial}{\partial x_l} u(\mathbf{x}; \mathbf{d}, \omega)\right)+{\omega^2} q(\mathbf{x}) u(\mathbf{x}; \mathbf{d}, \omega)}=0,\ \ \mathbf{x}\in\mathbb{R}^3,\\
\ \mbox{$u^+(\mathbf{x}):=u(\mathbf{x})-u^i(\mathbf{x})$ satisfies the Sommerfeld radiation condition. }
\end{cases}
\end{equation}
By the Sommerfeld radiation condition, we mean that the scattered wave $u^+(\mathbf{x})$ satisifies
\begin{equation}\label{eq:rad}
\lim_{|\mathbf{x}|\rightarrow +\infty}|\mathbf{x}|\left(\frac{\partial u^+(\mathbf{x})}{\partial |\mathbf{x}|}-i\omega u^+(\mathbf{x})\right)=0,
\end{equation}
which holds uniformly in $\hat{\mathbf{x}}:=\mathbf{x}/|\mathbf{x}|\in\mathbb{S}^2$. The radiation condition characterizes the decaying property of the scattered wave at infinity. The forward scattering problem \eqref{eq:Helm} has a unique solution $u\in H_{loc}^1(\mathbb{R}^3)$ (cf. \cite{Mcl}) and moreover, as $|\mathbf{x}|\rightarrow+\infty$ (cf. \cite{CK})
\begin{equation}\label{eq:scattering1}
u(\mathbf{x}; \mathbf{d}, \omega)=e^{i\omega \mathbf{x}\cdot \mathbf{d}}+\frac{e^{i\omega |\mathbf{x}|}}{|\mathbf{x}| } u_\infty(\hat{\mathbf{x}}, \mathbf{d})+\mathcal{O}\left(\frac{1}{|\mathbf{x}|^{2}}\right ),
\end{equation}
which holds uniformly in $\hat{\mathbf{x}}\in\mathbb{S}^2$. $u_\infty(\hat{\mathbf{x}},\mathbf{d})$ is known as the {\it scattering amplitude}, and it is real-analytic with resect to $\hat{\mathbf{x}}$ and $\mathbf{d}$, which are respectively referred to as the incident direction and observation angle. Clearly, $u_\infty$ depends on the incident wave field $e^{i\omega\mathbf{x}\cdot\mathbf{d}}$ as well as the acoustic medium $\{\Omega; \sigma, q\}$. In what follows, we shall also write $u_\infty(\hat{\mathbf{x}}, \mathbf{d}; \{\Omega; \sigma, q\})$ to signify such dependences, noting that we have dropped the dependence on $\omega$. $u_\infty(\hat{\mathbf{x}}, \mathbf{d}; \{\Omega; \sigma, q\})$ for all $(\hat{\mathbf{x}},\mathbf{d})\in\mathbb{S}^2\times\mathbb{S}^2$ carries the optical information of the acoustic inhomogeneity $\{\Omega;\sigma, q\}$. An important inverse scattering problem arising in practical applications is to recover the medium $\{\Omega; \sigma, q\}$ by knowledge of $u_\infty(\hat{\mathbf{x}},\mathbf{d})$.  This inverse problem is of fundamental importance to many areas of science and technology, such as radar and sonar, geophysical exploration, non-destructive testing, and medical imaging to name just a few; see \cite{CK,Isa,U2} and the references therein. In this context, an ideal or perfect invisibility cloak is introduced as follows.

\begin{defn}\label{def:cloak}
Let $\Omega_e$ and $\Omega_i$ be bounded domains such that $\Omega_i\Subset\Omega_e$.
$\Omega_e\backslash\overline{\Omega_i}$ and $\Omega_i$ signify, respectively, the
cloaking region and the cloaked region. Let $\Lambda_o$ and $\Lambda_i$
be two subsets of $\mathbb{S}^{2}$. $\{\Omega_e\backslash\overline{\Omega_i};
\sigma_c, q_c\}$ is said to be an (ideal/perfect) {\it invisibility
cloak} for the region $\Omega_i$ if
\begin{equation}\label{eq:definvisibility}
u_\infty\left(\hat{\mathbf{x}},\mathbf{d}; \{\Omega_e;\sigma_e,q_e\}\right)=0\quad \mbox{for}\ \
\hat{\mathbf{x}}\in\Lambda_o,\ \mathbf{d}\in\Lambda_i,
\end{equation}
where the extended medium is given as
\begin{equation}\label{eq:extended}
\{\Omega_e;\sigma_e,q_e\}=\begin{cases}
\ \sigma_a, q_a\quad & \mbox{in\ \ $\Omega_i$},\\
\ \sigma_c,q_c\quad & \mbox{in\ \ $\Omega_e\backslash\overline{\Omega_i}$},
\end{cases}
\end{equation}
with $\{\Omega_i; \sigma_a,q_a\}$ denoting a target medium located inside $\Omega_i$. If $\Lambda_o=\Lambda_i=\mathbb{S}^2$, then it is called a {\it
full cloak}, otherwise it is called a {\it partial cloak} with
limited apertures $\Lambda_o$ of observation angles, and $\Lambda_i$ of
impinging angles.
\end{defn}

By Definition~\ref{def:cloak}, one has that the cloaking layer $\{\Omega_e\backslash \overline{\Omega_i}; \sigma_c, q_c\}$ makes the target medium $\{\Omega_i;\sigma_a, q_a\}$ invisible to the exterior scattering measurements when the detecting waves come from the aperture $\Lambda_i$ and the observations are made in the aperture $\Lambda_o$. From a practical viewpoint, the cloaking device should be independent of the target medium $\{\Omega_i; \sigma_a, q_a\}$. Indeed, throughout the present study, the target medium $\{\Omega_i; \sigma_a, q_a\}$ could be arbitrary (but regular). As discussed earlier in the introduction, the cloaking medium $\{\Omega_e\backslash\overline{\Omega_i}; \sigma_c, q_c\}$ for an ideal cloak are usually singular; namely, the regular condition \eqref{eq:regular} is violated. Next, we introduce the definition of an approximate cloak by making use of regularized cloaking medium $\{\Omega_e\backslash\overline{\Omega_i}; \sigma_c^\delta, q_c^\delta\}$, which is the main theme of the current work. Here, $\delta\in\mathbb{R}_+$ signifies a regularization parameter, and $\{\Omega_e\backslash\overline{\Omega_i}; \sigma_c^\delta, q_c^\delta\}$ approaches $\{\Omega_e\backslash\overline{\Omega_i}; \sigma_c, q_c\}$ as $\delta\rightarrow +0$ in a certain sense.

\begin{defn}\label{def:cloak2}
Let $\Omega_e, \Omega_i$ and $\Lambda_o, \Lambda_i$ be as introduced in Definition~\ref{def:cloak}. Let $\{\Omega_e; \sigma_e^\delta, q_e^\delta\}$ be as defined in \eqref{eq:extended} with $\{\Omega_e\backslash \overline{\Omega_i}; \sigma_c, q_c\}$ replaced by $\{\Omega_e\backslash\overline{\Omega_i}; \sigma_c^\delta, q_c^\delta\}$. Here, $\{\Omega_e\backslash\overline{\Omega_i}; \sigma_c^\delta, q_c^\delta\}$ is as described earlier, and it is called an approximate cloak with incident aperture $\Lambda_i$ and observation aperture $\Lambda_o$ for the region $\Omega_i$ if
\begin{equation}\label{eq:definvisibility2}
u_\infty\left(\hat{\mathbf{x}},\mathbf{d}; \{\Omega_e;\sigma_e^\delta,q_e^\delta\}\right)=\mathcal{E}(\delta)\quad \mbox{for}\ \
\hat{\mathbf{x}}\in\Lambda_o,\ \mathbf{d}\in\Lambda_i,
\end{equation}
where $\mathcal{E}(\delta)$ is a positive function satisfying $\mathcal{E}(\delta)\rightarrow 0$ as $\delta\rightarrow+0$.
\end{defn}

Clearly, in Definition~\ref{def:cloak2}, $\mathcal{E}(\delta)$ characterizes the accuracy of the regularized approximate cloak. In our subsequent study, we sometime do not explicitly specify whether a cloak is ideal or approximate, and it should be clear from the context. In order to construct an invisibility cloak through the transformation-optics approach, a critical ingredient is the following transformation properties of the optical parameters (cf. \cite{GKLU3,LiLiuRonUhl}).
\begin{lem}\label{lem:transoptics}
Let $\Omega$ and $\widetilde\Omega$ be two Lipschitz domains, and suppose there exists a bi-Lipschitz and orientation-preserving mapping
\[
\widetilde{\mathbf{x}}=F(\mathbf{x}): \Omega\rightarrow\widetilde\Omega.
\]
Let $(\Omega; \sigma, q)$ be an acoustic medium supported in $\Omega$. Define the {\it push-forwarded} medium as follows,
\begin{equation}\label{eq:pushforward}
\{\widetilde\Omega; \widetilde\sigma, \widetilde q\}:=F_*\{\Omega; \sigma, q\},
\end{equation}
where
\begin{equation}
\begin{split}
&\widetilde{\sigma}(\widetilde{\mathbf{x}})=F_*\sigma(\widetilde{\mathbf{x}}):=\left({DF(\mathbf{x})\cdot \sigma(\mathbf{x})\cdot DF(\mathbf{x})^T}\right){|\mbox{\emph{det}}(DF(\mathbf{x}))|}^{-1}\bigg|_{\mathbf{x}=F^{-1}(\widetilde{\mathbf{x}})},\\
&\widetilde{q}(\widetilde{\mathbf{x}})=F_*q(\widetilde{\mathbf{x}}):={q(\mathbf{x})}{|\mbox{\emph{det}}(DF(\mathbf{x}))|}^{-1}\bigg|_{\mathbf{x}=F^{-1}(\widetilde{\mathbf{x}})},
\end{split}
\end{equation}
where $DF$ denotes the Jacobian matrix of $F$.
Then $u\in H^1(\Omega)$ solves the Helmholtz equation
\[
\nabla\cdot(\sigma\nabla u)+\omega^2 q u=0\quad \mbox{in\ $\Omega$},
\]
if and only if the pull-back field $\widetilde u=(F^{-1})^*u:=u\circ F^{-1}\in H^1(\widetilde\Omega)$ solves
\[
\widetilde\nabla\cdot(\widetilde\sigma\nabla \widetilde u)+\omega^2\widetilde q \widetilde u=0\quad \mbox{in\ $\widetilde\Omega$}.
\]
We have made use of $\nabla$ and $\widetilde\nabla$ to distinguish the
differentiations respectively in $\mathbf{x}$- and $\widetilde{\mathbf{x}}$-coordinates.

As a direct consequence, one has that if $F:\overline\Omega\rightarrow\overline\Omega$ is a bi-Lipschitz mapping such that $F|_{\partial\Omega}=\text{Id}$, then
\begin{equation}\label{eq:transequal}
u_\infty(\hat{\mathbf{x}},\mathbf{d}; \{\Omega;\sigma, q\})=u_\infty(\hat{\mathbf{x}}, \mathbf{d}; F_*\{\Omega; \sigma, q\}).
\end{equation}
\end{lem}

As a general description, in our subsequent study, $\delta\in\mathbb{R}_+$ always denotes the small regularization parameter. Let $\Omega_e, \Omega_i$ be the bounded domains in Definitions~\ref{def:cloak} and \ref{def:cloak2}. We shall also need $D_{\delta/2}\Subset D_\delta\Subset\Omega_i$ and $\Omega_i\Subset\Omega_l\Subset\Omega_e$, which shall be explicitly specified in our subsequent study. $D_{\delta/2}$ and $D_\delta$ are the virtual domains, and $D_\delta$ actually plays the role of $\Gamma_\delta$ in our earlier discussion.  Throughout the rest of the paper, we assume that there exists an orientation-preserving and bi-Lipschitz mapping
$F_\delta:\overline{\Omega_e}\rightarrow \overline{\Omega_e}$ such that
\beq\label{eq:blow1}
F_\delta(D_{\delta/2})= \Omega_i,\ F_\delta(D_\delta\backslash D_{\delta/2})=\Omega_l\backslash\Omega_i,\ F_\delta(\overline{\Omega_e}\setminus D_\delta)=\overline{\Omega_e}\setminus \Omega_l \ \mbox{and} \ F_\delta|_{\p \Omega}=\mbox{Id}.
\eeq
We refer to \cite{LiLiuRonUhl} for the discussion on the so-called assembled-by-components strategy of explicitly constructing such a transformation for generic $D_\delta$ and $\Omega_e$.
In our study, an approximate cloaking structure $\{\Omega_e; \Gs, q\}$ in the physical space is obtained by pushing forward
a virtual structure $\{\Omega_e; \Gs_\delta, q_\delta\}$,
\beq\label{eq:struc1}
\{\Omega_e; \Gs_\delta, q_\delta\}=(F_\delta^{-1})_*(\Omega_e; \sigma, q)=
\left\{
\begin{array}{ll}
I, 1 & \mbox{in} \quad \Omega_e\setminus\overline{D_\delta}, \\
\Gs_l, q_l & \mbox{in} \quad D_\delta\setminus \overline{D_{\delta/2}}, \\
\Gs_a, q_a & \mbox{in} \quad D_{\delta/2},
\end{array}
\right.
\eeq
where $\Gs_a$, $q_a$ are arbitrary but regular. Set
$$
\{\Omega_e; \Gs, q\}=(F_\delta)_*\{\Omega_e; \Gs_\delta, q_\delta\}\quad\mbox{and}\quad \{D_{\delta/2}; \Gs_a, q_a\}= (F_\delta^{-1})_*\{\Omega_i; \Gs_a', q_a'\},
$$
which denote, respectively, the physical cloaking structure and the physical object being cloaked. Clearly, $\{\Omega_i; \Gs_a', q_a'\}$ is also arbitrary but regular. The design of
$\Gs_l$ and $q_l$ in the lossy layer $D_\delta\setminus \overline{D}_{\delta/2}$ is critical and will be explicitly given in what follows at the appropriate place.
By the transformation acoustics in Lemma \ref{lem:transoptics}, it is straightforward to show that in order to quantify a cloaking device $\{\Omega_e; \sigma, q\}$ in the physical space, it suffices to derive the approximate cloaking effect to the virtual structure $\{\Omega_e; \Gs^\delta, q^\delta\}$. In other words, it is sufficient to
show that the {scattering amplitude} $u^\delta_\infty(\hat{\Bx},\mathbf{d})$ to the following system in the virtual space
\beq\label{eq:sys1}
\left\{
\begin{array}{l}
\Delta u_\delta +\omega^2 u_\delta =0 \hspace*{2cm} \mbox{in} \quad\RR^3\setminus \overline{D_\delta}, \\
\nabla\cdot \Gs^\delta\nabla u_\delta + \omega^2 q^\delta =0 \hspace*{.95cm} \mbox{in} \quad D_\delta, \\
u^+:=u_\delta- u^i  \quad \mbox{satisfies the radiation condition } \eqnref{eq:rad},
\end{array}
\right.
\eeq
verifies \eqnref{eq:definvisibility2}. Hence, throughout the rest of our study, we shall focus on quantifying the scattering of the virtual system \eqref{eq:sys1}.

\section{Boundary layer potentials}\label{sect:3}

In this section, we introduce the boundary layer potential operators for our subsequent use, and we refer to \cite{Ammari5,CK,Mcl,nedelec,Tay} for more relevant discussions.
Let $G_\omega$ be the fundamental solution to $\Delta+\omega^2$ given by
\beq\label{eq:fund1}
G_\omega(\Bx)=-\frac{e^{i\omega|\Bx|}}{4\pi|\Bx|}.
\eeq
For $\omega=0$ the corresponding fundamental solution is denoted by $G(\Bx)$. Let $D$ be a Lipshitz domain.
Denote by $\Scal_D^\omega$ the single layer potential operator
\beq\label{eq:singl1}
\Scal_D^\omega[\phi](\Bx):=\int_{\p D} G_\omega (\Bx-\By) \phi(\By)\, d\Gs_y\quad \Bx\in \RR^3\setminus \p D.
\eeq
It is well-known that the single layer potential $\Scal_D^\omega$
satisfies the trace formula
\beq \label{eq:trace}
\frac{\p}{\p\nu}\Scal_D^\omega[\phi] \Big|_{\pm} = \left(\pm \frac{1}{2}I+
(\Kcal_{B}^\omega)^*\right)[\phi] \quad \mbox{on } \p D.
\eeq
Here and throughout the rest of the paper, we make use of the following notation. For a function $u$ defined on $\mathbb{R}^3\backslash\partial D$, we denote
\[
u|_{\pm}(\Bx)=\lim_{\tau\rightarrow +0} u(\Bx\pm \tau\cdot\nu(\Bx)),\quad \Bx\in\partial D,
\]
and
\[
\frac{\partial u}{\partial\nu}\bigg|_\pm(\Bx)=\lim_{\tau\rightarrow +0}\langle \nabla_\Bx u(\Bx\pm \tau\cdot\nu(\Bx)),\nu(\Bx)\rangle,\quad \Bx\in\partial D,
\]
if the limits exist, where $\mathbf\nu(\Bx)=\mathbf\nu_x$ is the unit outward normal vector to $\p D$ at $\Bx\in\partial D$. In \eqref{eq:trace}, $(\Kcal_{D}^\omega)^*$ is the $L^2$-adjoint of $\Kcal_D^\omega$ and
$$
\Kcal_D^\omega [\phi]:= \mbox{p.v.} \quad \int_{\p D} \frac{\p G_\omega(\Bx-\By)}{\p \nu(\By)} \phi(\By)\, d\Gs_y , \quad \Bx\in \p D
$$
where p.v. signifies the Cauchy principle value.
We also define the double layer potential operator by $\Dcal_D^\omega$,
\beq\label{eq:dbl1}
\Dcal_D^\omega[\phi](\Bx):=\int_{\p D} \frac{\p G_\omega (\Bx-\By)}{\p \mathbf\nu_y} \phi(\By)\, d\Gs_y \quad \Bx\in \RR^3\setminus \p D
\eeq
and we have the following trace formula
\beq \label{eq:trace2}
\Dcal_D^\omega[\phi] \Big|_{\pm} = \left(\mp \frac{1}{2}I+
\Kcal_{B}^\omega\right)[\phi] \quad \mbox{on } \p D.
\eeq

By using the layer potential operators introduced above, we make use of the following integral representation to the solution of \eqnref{eq:sys1}
outside of $D_\delta$
\beq\label{eq:repre1}
u_\delta=u^i+ \Scal_{D_\delta}^\omega[\phi] \quad \mbox{in} \quad \RR^3\setminus \overline{D_\delta},
\eeq
where it is easily verified that the density function $\phi$ satisfies
\beq\label{eq:trans1}
\left(\frac{I}{2} + (\Kcal_{D_\delta}^\omega)^*\right)[\phi]= \frac{\p (u_\delta-u^i)}{\p \mathbf\mathbf\nu}\Big|_{\p D_\delta}^+.
\eeq

\section{Regularized full-cloak}\label{sect:4}

In this section, we consider the regularized cloaking construction by taking the generating set to be a curve. Let $\GG_0$ denote a smooth simple and non-closed curve in $\RR^3$ with two endpoints, denoted by $P_0$ and $Q_0$, respectively. Let $r\in\mathbb{R}_+$. Next, by using $\Gamma_0$ as a generating set, we shall construct a simply connected set $D_r$ in $\mathbb{R}^3$ for our cloaking study; see Fig.~\ref{fig1} for a schematic illustration. Denote by $N(\Bx)$ the normal plane of the curve $\GG_0$ at $\Bx\in\GG_0$. We note that $N(P_0)$ and $N(Q_0)$ are, respectively, defined by the left and right limits along $\Gamma_0$. For any $\Bx\in\GG_0$, we let $\mathscr{S}_r(\Bx)$ denote the disk lying on $N(\Bx)$, centered at $\Bx$ and of radius $r$. It is assumed that there exists $r_0\in\mathbb{R}_+$ such that when $r\leq r_0$, $\mathscr{S}_r(\Bx)$ intersects $\GG_0$ only at $\Bx$. We start with a thin structure $D_r^f$ given by
\begin{equation}\label{eq:drf}
D_r^f:=\mathscr{S}_r(\Bx)\times\Gamma_0(\Bx),\ \Bx\in\overline{\Gamma}_0,
\end{equation}
where we identify $\Gamma_0$ with its parametric representation $\Gamma_0(\Bx)$. Clearly, the facade of $D_r^f$, denoted by $S_r^f$ and parallel to $\GG_0$, is given by
\beq\label{eq:facade}
S_r^f:=\{\Bx+r\cdot \mathbf{n}(\Bx); \Bx\in \GG_0, \mathbf{n}(\Bx)\in N(\Bx) \cap \mathbb{S}^2\},
\eeq
and the two end-surfaces of $D_r^f$ are the two disks $\mathscr{S}_r(P_0)$ and $\mathscr{S}_r(Q_0)$. Let $D_{r_0}^a$ and $D_{r_0}^b$ be two simply connected sets with $\partial D_{r_0}^a=S_{r_0}^a\cup \mathscr{S}_{r_0}(P_0)$ and $\partial D_{r_0}^b=S_{r_0}^b\cup\mathscr{S}_{r_0}(Q_0)$. It is assumed that $S_{r_0}:=S^f_{r_0}\cup S^b_{r_0} \cup S^a_{r_0}$ is a $C^3$-smooth boundary of the domain $D_{r_0}:=D_{r_0}^a\cup D_{r_0}^f\cup D_{r_0}^b$. For $0<r<r_0$, we set
\[
D_r^a:=\frac{r}{r_0}(D_{r_0}^a-P_0)+P_0=\left\{\frac{r}{r_0}\cdot(\Bx-P_0)+P_0; \Bx\in D_{r_0}^a\right\},
\]
and similarly, $D_r^b:={r}/{r_0}\cdot(D_{r_0}^b-Q_0)+Q_0$. Let $S_r^a$ and $S_r^b$, respectively, denote the boundaries of $D_r^a$ and $D_r^b$ excluding $\mathscr{P}_r^a$ and $\mathscr{S}_r^b$. Now, we set $D_r:=D_r^a\cup D_r^f\cup D_r^b$, and $S_r:=S^f_r\cup S^b_r \cup S^a_r=\partial D_r$.

According to our earlier description, it is obvious that for $0<r\leq r_0$, $D_r$ is a simply connected set with $C^3$-smooth boundary $S_r$,  and $D_{r_1}\Subset D_{r_2}$ if $0\leq r_1<r_2\leq r_0$. Moreover,  $D_r$ degenerates to $\Gamma_0$ if one takes $r=0$. In order to ease the exposition, we assume that $r_0\equiv 1$. In what follows, we let $\delta\in\mathbb{R}_+$ be the asymptotically small regularization parameter and let
\beq\label{eq:rgn1}
S_\delta:=S^f_\delta \cup S^b_\delta \cup S^a_\delta,
\eeq
denote the boundary surface of the virtual domain $D_\delta$ used for the blowup construction. In order to ease the exposition, we drop the dependence on $r$ if one takes $r=1$. For example, $D$ and $S$ denote, respectively, $D_r$ and $S_r$ with $r=1$. It is emphasized that in all our subsequent argument, $D$ can always be replaced by $D_{\tau_0}$ with $0<\tau_0\leq r_0$ being a fixed number. Hence, we indeed shall not lose any generality of our study by assuming that $r_0\equiv 1$.
Finally, we would like to note that
\begin{figure}
\begin{center}
  \includegraphics[width=4.5in,height=2.0in]{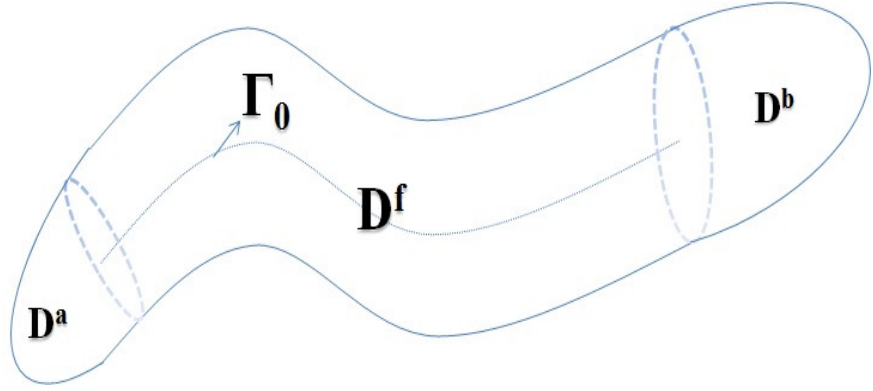}
  \end{center}
  \caption{Schematic illustration of the domain $D_r$ for the regularized full-cloak. \label{fig1}}
\end{figure}
a particular case is to take $\Gamma_0$ to be a straight line-segment and,
$S^{b}_r$ and $S^{a}_r$ to be two semi-spheres of radius $r$ and centered at $P_0$ and $Q_0$ respectively, which is exactly the one considered in \cite{LiLiuRonUhl}. Hence, the geometry in our study is much more general than that considered in \cite{LiLiuRonUhl}; see also Remark~\ref{rem:general} in what follows for more relevant discussions.

We are now in a position to present the main theorem on the approximate cloak constructed by using $D_\delta$ described above as the virtual domain.
\begin{thm}\label{th:main1}
Let $D_\delta$ be as described above with $\p D_\delta=S_\delta$ defined in \eqnref{eq:rgn1}.
Let $u_\delta$ be the solution to \eqnref{eq:sys1} with $\{\Omega_e; \Gs_\delta, q_\delta\}$
defined in \eqnref{eq:struc1}, and $\{D_\delta\backslash\overline{D_{\delta/2}}; \sigma_l, q_l\}$ given in \eqnref{eq:loss1}. Let $u_\infty^\delta(\hat\Bx,\mathbf{d})$ denote the scattering amplitude of $u_\delta$. Then there exists $\delta_0\in\mathbb{R}_+$ such that when $\delta<\delta_0$,
\beq\label{eq:ncest1}
|u^\delta_\infty(\hat\Bx,\mathbf{d})|\leq C(\omega) \delta^2,
\eeq
where $C(\omega)$ is a positive constant depending on $\omega$ and $D$, but independent of $\Gs_a$, $q_a$ and $\hat\Bx$, $\mathbf{d}$.
\end{thm}

\begin{rem}\label{lem:physical1}
Following our discussion at the end of Section~2, one can immediately infer by Theorem~\ref{th:main1} that the push-forwarded structure
$$\{\Omega_e; \Gs, q\}=(F_\delta)_*\{\Omega; \Gs_\delta, q_\delta\}$$
with $F_\delta$ defined by \eqnref{eq:blow1} and $\{\Omega_e; \Gs_\delta, q_\delta\}$ defined in the theorem, produces an approximate cloaking device within $\delta^2$-accuracy to the ideal cloak. Indeed, if one lets $u_\infty$ denote the scattering amplitude corresponding to the cloaking structure, then there holds
\beq\label{eq:ncest1}
|u_\infty(\hat\Bx,\mathbf{d})|\leq C(\omega) \delta^2
\eeq
where $C(\omega)$ is independent of $\Gs_a'$, $q_a'$ and $\hat\Bx$, $\mathbf{d}$.
\end{rem}

\begin{rem}\label{lem:sharpness1}
It is remarked that our estimate derived in Theorem~\ref{th:main1} is sharp and this can be verified by the numerical examples presented in \cite{LiLiuRonUhl}.
\end{rem}

The subsequent three subsections are devoted to the proof of Theorem~\ref{th:main1}. For our later use, we first derive some critical lemmas. In what follows, we let $\Bz$ denote the space variable on $\Gamma_0$ and for every $\By\in D_r^f$ we define a new variable $\Bz_y\in \GG_0$ which is the projection of $\By$ onto $\GG_0$.
Meanwhile, if $\By$ belongs to $D_r^a$ (respectively $D_r^b$) then $\Bz_y$ is defined to be $P_0$ (respectively $Q_0$).
Let $t$ denote the arc-length parameter of $\GG_0$ and $\theta$,
which ranges from $0$ to $2\pi$, be the angle of the point on $N(t)$ with respect to the centeral point $\Bx(t)\in \Gamma_0$. Moreover, we assume that if $\theta=0$, then the corresponding points are those lying on the line that connects $\GG_0(t)$ to $\GG_1(t)$, where
$\GG_1$ is defined to be
$$
\GG_1:=\{\Bx + \mathbf{n}_1(\Bx); \Bx \in \GG_0, \mathbf{n}_1(\Bx)\in N(\Bx)\cap \mathbb{S}^2-\mbox{a fixed vector for a given $\Bx$}\}.
$$

Next we introduce a blowup transformation which maps $\By\in\overline{D_\delta}$ to $\tilde\By\in \overline{D}$ as follows
\begin{equation}\label{eq:A}
A(\By)=\tilde\By:=\frac{1}{\delta}(\By-\Bz_y)+\Bz_y,\quad \By\in D_\delta^f,
\end{equation}
whereas
\beq\label{eq:Ares1}
A(\By)=\tilde\By:= \left\{
\begin{array}{ll}
\frac{\By-P_0}{\delta}+P_0, & \By \in D_\delta^a, \\
\frac{\By-Q_0}{\delta}+Q_0, & \By \in D_\delta^b.
\end{array}
\right.
\eeq
Then, we can show that

\begin{lem}\label{le:2}
Let $A$ be the transformation introduced in \eqref{eq:A} and \eqref{eq:Ares1} which maps the region $\overline{D_\delta}$ to $\overline{D}$.
Let $B(\By)$ be the corresponding Jacobi matrix of $A(\By)$ given by
\begin{equation}\label{eq:jacob1}
B(\By)=\nabla_\By A(\By),\quad \By\in\overline{D_\delta}.
\end{equation}
Then we have $B(\By)^T=B(\By)$ and
\beq\label{eq:le22}
 B(\By)\mathbf\nu_y= \frac{1}{\delta}\mathbf\nu_y,\quad \By\in\partial D_\delta,\\
\eeq
where $\nu_y$ stands for the unit outward normal vector to $\p D_\delta$.
\end{lem}
\begin{proof}
It is easily seen that if $\By\in D_\delta^a\cup D_\delta^b$, then
$$
B(\By)=\frac{1}{\delta}\, \mathbf{I}_{3\times 3} , \quad \By\in D_\delta^a \cup D_\delta^b.
$$
Next, we let $\By\in D_\delta^f$ and note that $\By-\Bz_y=\delta(\tdy-\Bz_{\tilde{y}})$, namely
$$
A(\By)=\tdy=\frac{1}{\delta}\By-(\frac{1}{\delta}-1)\Bz_{y}.
$$
Hence it follows
\beq\label{tmp2}
B(\By)=\nabla_y A(\By)=\frac{1}{\delta}\mathbf{I}_{3\times 3}-(\frac{1}{\delta}-1)\Bz_y'(t)\Bz_y'(t)^T,
\eeq
where the superscript $T$ denotes the transpose of a vector or matrix.
Then it is straightforward to see that
$$
B(\By)\mathbf\nu_y=\frac{1}{\delta}\mathbf\nu_{y},\quad \By\in\partial D_\delta.
$$

The proof is complete.
\end{proof}

Next, we present the crucial design of the lossy layer.
Define the material parameters $\Gs_\delta$ and $q_\delta$ in the lossy layer $D_\delta\setminus\overline{D}_{\delta/2}$ to be
\beq\label{eq:loss1}
\Gs_\delta(\Bx)=\Gs_l(\Bx):= \gamma\cdot B^{-2}(\Bx), \quad q_\delta(\Bx)=q_l(\Bx): = \alpha + i \beta \quad \mbox{for} \quad \Bx\in D_\delta\setminus\overline{D}_{\delta/2},
\eeq
where $\gamma$, $\alpha$ and $\beta$ are all positive real numbers. It is easily seen from \eqnref{tmp2} that $B$ is a symmetric positive definite matrix and hence $\Gs_\delta^{-1}$ is a well-defined regular material tensor.

\subsection{Asymptotic expansions}
In order to tackle the integral equation \eqref{eq:trans1}, we shall first derive some crucial asymptotic expansions. Henceforth, we denote $\tilde{\phi}(\tdy):=\phi(\By)$ for $\tdy=A(\By)$, $\By\in D_\delta$
and $\tdy\in \p D$. The same notation shall be adopted for $\Bx$ and $\tdx$. First, we note that
\beq
d\Gs_y=\left\{
\begin{array}{ll}
\delta\, d\Gs_{\tilde{y}}, & \By\in S_\delta^f,\\
\delta^2\, d\Gs_{\tilde{y}}, & \By\in S_\delta^b\cup S_\delta^a.
\end{array}
\right.
\eeq
For a function $u$, its gradient in $D_\delta^f$ can be decomposed into two parts: the tangential derivative
with respect to $\GG_0$ and the corresponding normal derivative as follows
$$
\nabla_y u(\By)= \p_t u(\By) T_y+ \nabla_l u(\By),
$$
where $T_y$ is the tangential direction along $\Gamma_0$. Here, $\nabla_l$ signifies the normal
derivative.
Then for $\Bx\in \RR^3\setminus D_\delta$ with sufficient large $|\Bx|$, we can expand $\Scal_{D_\delta}^{\omega}[\phi]$ in
 $\Bz\in\Gamma_0$ as follows
\begin{align}
\Scal_{D_\delta}^{\omega}[\phi](\Bx)=&\int_{\p D_\delta}G_\omega(\Bx-\By)\phi(\By)\, d\Gs_y \nonumber\\
=& \int_{\p D_\delta}\Big(G_\omega(\Bx-\Bz_y)-\nabla G_\omega(\Bx-\Bz_y)\cdot(\By-\Bz_y)\Big)\phi(\By)\, d\Gs_y \nonumber\\
 & +\int_{\p D_\delta}(\By-\Bz_y)^T\nabla^2G_\omega(\Bx-\mathbf\zeta(\By))(\By-\Bz_y)\phi(\By)d\Gs_y \nonumber\\
=& \int_{\p D_\delta}G_\omega(\Bx-\Bz_y)\phi(\By)\, d\Gs_y -\int_{\p D_\delta}\nabla G_\omega(\Bx-\Bz_y)\cdot(\By-\Bz_y)\phi(\By)\, d\Gs_y
\nonumber\\
&+ \int_{\p D_\delta}(\By-\Bz_y)^T\nabla^2G_\omega(\Bx-\mathbf\zeta(\By))(\By-\Bz_y)\phi(\By)d\Gs_{y}\nonumber\\
:=& R_1 +R_2 +R_3 \label{eq:I1I2},
\end{align}
where $\zeta(\By)=\eta \By+(1-\eta)\Bz_y \in D_\delta$ for some $\eta\in (0,1)$, and the superscript $T$ signifies the matrix transpose.
The term $R_3$ in \eqref{eq:I1I2} is a remainder term from the Taylor series expansion and it verifies the following estimate
\begin{equation}\label{eq:R3}
 |R_3|=\delta^2 \Big|\int_{\p D_\delta}(\tdy-\Bz_{\tilde{y}})^T\nabla^2G_\omega(\Bx-\mathbf\zeta(\By))(\tdy-\Bz_{\tilde{y}})\phi(\By)d\Gs_{y}\Big|
\leq \delta^3 \frac{1}{|\Bx|} \|\tilde\phi\|_{H^{-3/2}(\p D)}.
\end{equation}
We shall also need the expansion of the incident plane wave $u^i$ in $\Bz\in\Gamma_0$, and there holds
\beq\label{eq:exp1}
u^i(\By)=u^i(\Bz_y)+\nabla u^i(\Bz_y)\cdot (\By-\Bz_y)+ \sum_{|\alpha|=2}^\infty \p_y^\alpha u^i(\Bz_y)(\By-\Bz_y)^\alpha,
\eeq
where the multi-index $\alpha=(\alpha_1, \alpha_2, \alpha_3)$ and $\p_y^{\alpha}=\p_{y_1}^{\alpha_1}\p_{y_2}^{\alpha_2}\p_{y_3}^{\alpha_3}$.
Since for $\By\in \p D_\delta$, $\mathbf\nu_y=\mathbf\nu_{\tilde{y}}$, one further has
$$
\mathbf\nu_y\cdot\nabla_y u^i(\By)= \mathbf\nu_y\cdot\nabla\sum_{|\alpha|=0}^{\infty} \p_y^\alpha u^i(\Bz_y)(\By-\Bz_y)^\alpha=
\sum_{|\alpha|=1}^{\infty}\delta^{|\alpha|-1} \p_y^\alpha u^i(\Bz_y)\mathbf\nu_{\tilde{y}}\cdot\nabla(\tdy-\Bz_y)^\alpha.
$$

The following lemma is of critical importance for our subsequent analysis.

\begin{lem}\label{le:1}
Let $\phi$ be the solution to \eqnref{eq:trans1} and $\tilde\phi(\tdx)=\phi(\Bx)$ for $\Bx\in\partial D_\delta$ and $\tilde\Bx\in\partial D$. There hold the following results
\beq\label{eq:arg1}
(\Kcal_{D_\delta}^\omega)^*[\phi](\Bx)=
\left\{
\begin{array}{ll}
\delta (\Kcal_{S^f}^\omega)^*[\tilde\phi](\tdx)+ \Ocal(\delta^2\|\tilde\phi\|_{H^{-3/2}(\p D)}), & \Bx \in S_\delta^f, \\
(\Kcal_{S^c})^*[\tilde\phi](\tdx)+ \Ocal(\delta\|\tilde\phi\|_{H^{-3/2}(\p D)})& \Bx \in S_\delta^c, \quad c\in\{a, b\},
\end{array}
\right.
\eeq
and
\beq\label{eq:arg2}
\Scal_{D_\delta}^\omega[\phi](\Bx)=\left\{
\begin{array}{ll}
\delta \Scal_{S^f}^\omega[\tilde\phi](\tdx)+ \Ocal(\delta^2\|\tilde\phi\|_{H^{-3/2}(\p D)}), &\Bx \in S_\delta^f, \\
\delta (\Scal_{S^c}+\Scal_{S^f}^\omega)[\tilde\phi](\tdx)+ \Ocal(\delta^2\|\tilde\phi\|_{H^{-3/2}(\p D)}), & \Bx \in S_\delta^c, \quad c\in\{a, b\},
\end{array}
\right.
\eeq
where
\beq\label{eq:Kcal411}
(\Kcal_{S^f}^\omega)^*[\tilde\phi](\tdx):=\frac{1}{4\pi}\int_{\GG_0}\Big(\frac{\la \Bz_{\tilde{x}}-\Bz_{\tilde{y}},\mathbf\nu_x\ra}{|\Bz_{\tilde{x}}-\Bz_{\tilde{y}}|^3}
-i\omega\frac{\la \Bz_{\tilde{x}}-\Bz_{\tilde{y}},\mathbf\nu_x\ra}{|\Bz_{\tilde{x}}-\Bz_{\tilde{y}}|^2}\Big)
e^{i\omega|\Bz_{\tilde{x}}-\Bz_{\tilde{y}}|}\int_0^{2\pi}\tilde\phi(\tdy)\, d\theta_{\tilde{y}}\, ds_t ,
\eeq
and
$$
(\Kcal_{S^c})^*[\tilde\phi](\tdx):=\frac{1}{4\pi}\int_{S^c} \frac{\la \tilde{\Bx}-\tilde{\By},\mathbf\nu_x\ra}{|\tilde{\Bx}-\tilde{\By}|^3}
\tilde\phi(\tdy)\, d\Gs_{\tilde{y}},
$$
and
$$
\Scal_{S^f}^\omega[\tilde\phi](\tdx):=-\frac{1}{4\pi}\int_{\GG_0}\frac{e^{i\omega|\Bz_{\tilde{x}}-\Bz_{\tilde{y}}|}}{|\Bz_{\tilde{x}}-\Bz_{\tilde{y}}|}
\int_0^{2\pi}\tilde\phi(\tdy)d\theta_{\tilde{y}}ds_t, \quad \Scal_{S^c}[\tilde\phi](\tdx):=\int_{S^c} G(\tdx-\tdy) \tilde\phi(\tdy)\, d\Gs_{\tilde{y}}.
$$
The variables $\Bz_{\tilde{x}}, \Bz_{\tilde{y}}$ are in $\GG_0$.
\end{lem}

\begin{proof}
For
$\Bx, \By\in D_\delta$, one has
$$
|\Bx-\By|=|(\Bx-\Bz_x)-(\By-\Bz_y)+\Bz_x-\Bz_y|=|\delta((\tdx-\Bz_{\tilde{x}})-(\tdy-\Bz_{\tilde{y}}))+(\Bz_{\tilde{x}}-\Bz_{\tilde{y}})|
$$
Hence, we have the following expansion for $\Bx\in S_\delta^f$,
$$
|\Bx-\By|=|\Bz_{\tilde{x}}-\Bz_{\tilde{y}}|+\delta \la(\tdx-\Bz_{\tilde{x}})-(\tdy-\Bz_{\tilde{y}}), \frac{\Bz_{\tilde{x}}-\Bz_{\tilde{y}}}{|\Bz_{\tilde{x}}-\Bz_{\tilde{y}}|}\ra + \Ocal(\delta^2).
$$
Similarly
\begin{align*}
&\la \Bx-\By, \mathbf\nu_x\ra = \la \Bz_{\tilde{x}}-\Bz_{\tilde{y}}, \mathbf\nu_x\ra +\delta \la (\tdx-\Bz_{\tilde{x}})-(\tdy-\Bz_{\tilde{y}}), \mathbf\nu_x\ra,\\
&|\Bx-\By|^{-1}=|\Bz_{\tilde{x}}-\Bz_{\tilde{y}}|^{-1}-\delta \la \frac{\Bz_{\tilde{x}}-\Bz_{\tilde{y}}}{|\Bz_{\tilde{x}}-\Bz_{\tilde{y}}|^3}, (\tdx-\Bz_{\tilde{x}})-(\tdy-\Bz_{\tilde{y}})\ra + \Ocal(\delta^2),\\
&e^{i\omega |\Bx-\By|}=e^{i\omega|\Bz_{\tilde{x}}-\Bz_{\tilde{y}}|}\Big(1+i\omega\delta\la (\tdx-\Bz_{\tilde{x}})-(\tdy-\Bz_{\tilde{y}}), \frac{\Bz_{\tilde{x}}-\Bz_{\tilde{y}}}{|\Bz_{\tilde{x}}-\Bz_{\tilde{y}}|}\ra\Big)
+\Ocal(\delta^2).
\end{align*}
With those expansions at hand, we can compute for $\Bx\in S_\delta^f$,
\begin{align*}
&(\Kcal_{D_\delta}^\omega)^*[\phi](\Bx)=\frac{1}{4\pi}\int_{S_\delta^f} \Big(\frac{1}{|\Bx-\By|}-i\omega\Big)
\frac{\la \Bx-\By, \mathbf\nu_x\ra}{|\Bx-\By|^2}e^{i\omega |\Bx-\By|}\phi(\By)\, d\Gs_y+\Ocal(\delta^2\|\tilde\phi\|_{H^{-3/2}(\p D)})\\
=& \delta \frac{1}{4\pi}\int_{\GG_0}(|\Bz_{\tilde{x}}-\Bz_{\tilde{y}}|^{-1}-i\omega)\frac{\la \Bz_{\tilde{x}}-\Bz_{\tilde{y}}, \mathbf\nu_x\ra}{|\Bz_{\tilde{x}}-\Bz_{\tilde{y}}|^2}
e^{i\omega|\Bz_{\tilde{x}}-\Bz_{\tilde{y}}|}\int_0^{2\pi}\tilde\phi(\tdy)\, d\theta_{\tilde{y}}\, ds_t +\Ocal(\delta^2\|\tilde\phi\|_{H^{-3/2}(\p D)}),
\end{align*}
which proves the case of \eqnref{eq:arg1} for $\Bx\in S_\delta^f$. In a similar manner, one can prove the first case
of \eqnref{eq:arg2} with $\Bx\in S_\delta^f$. Next, we note that if $\Bx\in S_\delta^c$, $c\in \{a ,b \}$, then
$$
\Bz_{\tilde{x}}-\Bz_{\tilde{y}}=0, \quad \mbox{for} \quad \By\in S_\delta^c
$$
and
$$
|\Bx-\By|=\delta|\tdx-\tdy|,\quad \frac{\la\Bx-\By,\mathbf\nu_x\ra}{|\Bx-\By|}=\frac{\la\tdx-\tdy,\mathbf\nu_x\ra}{|\tdx-\tdy|}, \quad \Bx,\By\in S_\delta^c.
$$
Then by using a similar argument to the proof of the first case of \eqnref{eq:arg1}, we can prove \eqnref{eq:arg1} for $\Bx\in S_\delta^c$, $c\in \{a ,b \}$.
Finally, the second case in \eqnref{eq:arg2} can also be proved following a similar argument by noting that
$$
|\Bx-\By|=\delta|\tdx-\tdy|, \quad \Bx,\By\in S_\delta^c.
$$
The proof is complete.
\end{proof}

\subsection{An important result for acoustic scattering}
In this subsection, using the results obtained in the previous subsection, we present a theorem concerning the scattering from a thin scatterer. The asymptotic expansion formula derived in the next theorem would find important applications in inverse scattering theory. Indeed, it might be used to devise some novel inverse scattering scheme for recovering a thin sound-hard obstacle or a thin acoustic medium with a uniform content. Hence, the result is of significant practical and theoretical interests for its own sake. To our best knowledge, there is no available result in the literature of this form. Moreover, \eqref{eq:thmprin1} derived in the theorem shall be needed in our subsequent proof of Theorem~\ref{th:main1} concerning the approximate cloak with an arbitrary content being cloaked.

\begin{thm}\label{prop:2}
Let $u_\delta$ be the solution to \eqnref{eq:sys1}. Define
$\tilde\Phi(\tdy):=\Phi(\By)=\frac{\p u_\delta(\By)}{\p \mathbf\nu_y}\Big|_{\partial D_\delta}^+$,
then there holds for $\Bx\in\mathbb{R}^3\backslash\overline{D}$,
\beq\label{eq:thmprin1}
\int_{S^f} G_\omega(\Bx-\Bz_{\tilde{y}}) \tilde\phi(\tdy)\, d\Gs_{\tilde{y}} = 2\int_{S^f} G_\omega(\Bx-\Bz_{\tilde{y}})\tilde\Phi(\tdy)\, d\Gs_{\tilde{y}} +\Ocal(\delta(\|\tilde\Phi\|_{H^{-3/2}(S^f)}+1)).
\eeq
If one assumes that $\Phi(\By)=0$, $\By\in \p D_\delta$, namely $D_\delta$ is a sound-hard obstacle, then there holds for $\Bx\in \RR^3\setminus \overline{D}$
\begin{equation}\label{eq:aaa1}
\begin{split}
(u_\delta-u^i)(\Bx)=& 2\delta^2\int_{S^f}\nabla_l G_\omega(\Bx-\Bz_y)\cdot(\tdy-\Bz_y)\mathbf\nu_{\tilde{y}}\cdot \nabla u^i(\Bz_y) \, d\Gs_{\tilde{y}}\\
&-\delta^2\int_{S^f}G_\omega(\Bx-\Bz_y)\Delta_l u^i (\Bz_y)\, d\Gs_{\tilde{y}} \\
&-\delta^2\int_{S^a\cup S^b}G_\omega(\Bx-\Bz_y)\nu_{\tilde{y}}\cdot \nabla u^i (\Bz_y)d\Gs_{\tilde{y}}+\Ocal(\delta^3),
\end{split}
\end{equation}
where for a function $u$, $\Delta_l u=\nabla_l\cdot \nabla_l u$.
\end{thm}

\begin{proof}
By \eqnref{eq:trans1}, we see that
$$
\phi(\Bx)=\Big(\frac{I}{2}+(\Kcal_{D_\delta}^\omega)^*\Big)^{-1}\Big[\mathbf\nu_{\By}\cdot \nabla (u_\delta-u^i)(\By)\Big|_{\partial D_\delta}^+\Big](\Bx).
$$
Using the results in Lemma \ref{le:1}, along with the fact that (cf. \cite{nedelec,Tay})
\[
\frac I 2+(\mathcal{K}_{S^c})^*:\  H^{-3/2}(S^c)\rightarrow H^{-3/2}(S^c)\quad\mbox{is\ invertible},
\]
and the expansion of $u^i$ in \eqref{eq:exp1}, we have for $\Bx\in S_\delta^f$ that
\beq\label{eq:expanphi411}
\tilde\phi(\tdx)=2\tilde\Phi(\tdx)-2 \mathbf\nu_{\Bx}\cdot \nabla_l u^i(\Bz_{\tilde{x}}) + \Ocal(\delta(\|\tilde\Phi\|_{H^{-3/2}(S^f)}+1)).
\eeq
By using Green's formula
\begin{equation}\label{eq:ef1}
\int_0^{2\pi} \mathbf\nu\cdot \nabla_l u^i(\Bz)\, d\theta=\int_0^{2\pi} \mathbf\nu\, d\theta\cdot\nabla_l u^i(\Bz)
=\int_{\mathbb{B}(0,1)}\nabla_l\cdot \nabla_l u^i(\Bz)\, ds =0,
\end{equation}
where $\mathbb{B}(0,1)$ is the unit disk. Combining \eqref{eq:expanphi411} and \eqref{eq:ef1}, one can show \eqnref{eq:thmprin1}.

Next if $\tilde\Phi(\tdy)=\Phi(\By)=0$, it then can be seen from \eqref{eq:expanphi411} that
$$
\|\tilde\phi\|_{H^{-3/2}(\p D)}\leq C,
$$
where $C$ is a positive constant depending only on $D$.
Then by using our earlier results in \eqnref{eq:I1I2} and \eqref{eq:R3}, one can first show that
\begin{equation}\label{eq:ddd2}
\begin{split}
\Scal_{D_\delta}^{\omega}[\phi](\Bx)
=& \int_{\p D_\delta}G_\omega(\Bx-\Bz_y)\phi(\By)\, d\Gs_y -\int_{\p D_\delta}\nabla G_\omega(\Bx-\Bz_y)\cdot(\By-\Bz_y)\phi(\By)\, d\Gs_y +\Ocal(\delta^3) \\
=&\delta\int_{S^f}G_\omega(\Bx-\Bz_{\tilde{y}})\tilde{\phi}(\tdy)d\Gs_{\tilde{y}} +\delta^2\Big(\int_{S^a\cup S^b}G_\omega(\Bx-\Bz_{\tilde{y}})\tilde{\phi}(\tdy)\, d\Gs_{\tilde{y}}\\
&-\int_{S^f}\nabla_l G_\omega(\Bx-\Bz_{\tilde{y}})\cdot(\tdy-\Bz_{\tilde{y}})\tilde{\phi}(\tdy)
\, d\Gs_{\tilde{y}}\Big)+\Ocal(\delta^3).
\end{split}
\end{equation}
By using Lemma \ref{le:1} together with \eqref{eq:exp1} and \eqnref{eq:ddd2}, we have for $\Bx\in S_\delta^f$ that
\begin{align}
\tilde\phi(\tdx)=&-2 \mathbf\nu_{\Bx}\cdot \nabla_l u^i(\Bz_{\tilde{x}}) - 2\delta \mathbf\nu_{\Bx}\cdot \nabla_l\sum_{|\alpha|=2}
\p_l^\alpha u^i(\Bz_{\tilde{x}}) (\tdx-\Bz_{\tilde{x}})^{\alpha} \label{eq:tmpexp1}\\
&+ 4\delta(\Kcal_{S^f}^\omega)^*[\mathbf\nu_{\tilde{y}}\cdot \nabla_l u^i(\Bz_{\tilde{y}})](\tdx)+\Ocal(\delta^2).\nonumber
\end{align}
Taking the integral of \eqnref{eq:tmpexp1} with respect to $\theta$ from $0$ to $2\pi$ and using \eqnref{eq:Kcal411} and \eqnref{eq:ef1} one can easily obtain
$$
\int_0^{2\pi}\tilde\phi(\tdx)\, d\theta= -2\pi\delta \Delta_l u^i (\Bz_{\tilde{x}})+\Ocal(\delta^2), \quad \tdx \in S^f.
$$
Therefore,
\begin{equation}\label{eq:ddd4}
\int_{S^f}G_\omega(\tdx-\Bz_{\tilde{y}})\tilde{\phi}(\tdy)\, d\Gs_{\tilde{y}}=-\delta\int_{S^f}G_\omega(\Bx-\Bz_y)\Delta_l u^i(\Bz_y)\, d\Gs_{\tilde{y}}+\Ocal(\delta^2).
\end{equation}

Finally, by plugging \eqnref{eq:tmpexp1} and \eqref{eq:ddd4} into \eqref{eq:ddd2}, along with straightforward calculations, one can obtain \eqref{eq:aaa1} and thus complete the proof.
\end{proof}
\begin{rem}
Theorem \ref{prop:2} gives the asymptotic expansion of the scattered wave field from a thin sound-hard obstacle $D_\delta$ in terms of the asymptotic parameter $\delta\in\mathbb{R}_+$. Here, by a sound-hard obstacle, we mean a scatterer that the wave cannot penetrate inside and the wave velocity vanishes on the boundary of the scatterer (cf. \cite{CK}). The result can also be extended to deriving the asymptotic expansion of the scattered wave field if $\{D_\delta;\sigma_\delta,q_\delta \}$ is a uniform inhomogeneity; that is, both $\sigma_\delta$ and $q_\delta$ are two fixed constants.
Those results would find important applications in inverse
problems of reconstructing the scatterers; see \cite{iakovleva, ADM14} and references therein for related studies. Indeed, as is known that the corresponding inverse scattering problems are nonlinear and the asymptotic expansion formulas naturally give rise to their linearized counterparts.
\end{rem}


\subsection{Proof of the main theorem}
Let us go back to the proof of Theorem~\ref{th:main1} by continuing with the estimates of $R_1$ and $R_2$ in \eqnref{eq:I1I2} for the thin virtual cloaking structure with
arbitrary but regular $\Gs_a $ and $q_a$ in $D_{\delta/2}$. In what follows, we let $C$ denote a generic positive constant. It may change from one inequality to another inequality in our estimates. Moreover, it may depend on different parameters, but it is independent of $\sigma_a$, $q_a$ and $\hat\Bx$ and $\Bd$. We shall also write $C(\omega)$ to signify its dependence on the frequency $\omega$.

Define
\begin{equation}\label{eq:ppp}
\tilde\Phi(\tdy):=\Phi(\By)=\frac{\p u_\delta(\By)}{\p \mathbf\nu_y}\Big|_+.
\end{equation}
We first note that by using \eqref{eq:expanphi411} in Lemma \ref{le:1} there holds
\beq\label{eq:estphi11}
\|\tilde\phi\|_{H^{-3/2}(\p D)}\leq C (\|\tilde\Phi\|_{H^{-3/2}(\p D)} +1).
\eeq
Next, by taking expansion around $\Bz\in\Gamma_0$ and using \eqnref{eq:thmprin1} one can show that for $\delta\in\mathbb{R}_+$ sufficiently small and $|\Bx|$ sufficiently large
\begin{equation}\label{eq:ccc1}
\begin{split}
|R_1|=&\Big|\int_{\p D_\delta}G_\omega(\Bx-\Bz_y)\phi(\By)\, d\Gs_y\Big|\\
\leq& \delta\Big|\int_{S^f} G_\omega(\Bx-\Bz_{\tilde{y}})
\tilde{\phi}(\tdy)\, d\Gs_{\tilde{y}}\Big|+ C\frac{\delta^2}{|\Bx|}\|\tilde\phi\|_{H^{-3/2}(\p D)}\\
\leq&2\delta\Big|\int_{S^f} G_\omega(\Bx-\Bz_{\tilde{y}})
\tilde{\Phi}(\tdy)\, d\Gs_{\tilde{y}}\Big|+ C\delta^2\frac{1}{|\Bx|}(\|\tilde\Phi\|_{H^{-3/2}(\p D)}+1) \\
\leq & C\delta\frac{1}{|\Bx|} \|\tilde{\Phi}\|_{H^{-3/2}(S^f)}+ C\delta^2\frac{1}{|\Bx|}(\|\tilde\Phi\|_{H^{-3/2}(\p D)}+1).
\end{split}
\end{equation}
The estimate
of $R_2$ can also be done using Taylor's expansions and \eqnref{eq:expanphi411}
\begin{align}
|R_2|=& \Big|\int_{\p D_\delta}\nabla G_\omega(\Bx-\Bz_y)\cdot(\By-\Bz_y)\phi(\By)\, d\Gs_y\Big|\nonumber\\
\leq & 2\delta^2\Big|\int_{S^f}\nabla_l G_\omega(\Bx-\Bz_{\tilde{y}})\cdot (\tdy-\Bz_{\tilde{y}})\tilde\Phi(\tdy)\, d\Gs_{\tilde{y}}\nonumber\\
&+\int_{S^f}\nabla_l G_\omega(\Bx-\Bz_{\tilde{y}})\cdot(\By-\Bz_y) \nu_{\By}\cdot \nabla_l u^i(\Bz_{\tilde{y}})\, d\Gs_{\tilde{y}}\Big|
+ C\delta^3 (\|\tilde\Phi\|_{H^{-3/2}(\p D)}+1)
\nonumber\\
\leq & C\delta^2\frac{1}{|\Bx|}(\|\tilde\Phi\|_{H^{-3/2}(S^f)}+1)+C\delta^3\frac{1}{|\Bx|} (\|\tilde\Phi\|_{H^{-3/2}(\p D)}+1). \label{eq:sharp1}
\end{align}
Hence, by applying the estimates in \eqref{eq:R3}, \eqref{eq:ccc1} and \eqref{eq:sharp1} to \eqref{eq:I1I2} and \eqref{eq:repre1}, we readily have
\beq\label{eq:mainest1}
|u_\delta-u^i|\leq C\frac{\delta}{|\Bx|} \|\tilde{\Phi}\|_{H^{-3/2}(S^f)}+C\delta^2\frac{1}{|\Bx|}(\|\tilde\Phi\|_{H^{-3/2}(\p D)}+1)
\eeq
for $|\Bx|$ sufficiently large. Here, it is noted that we have made use of $\|\tilde{\Phi}\|_{H^{-3/2}(S^f)}$
instead of $\|\tilde{\Phi}\|_{H^{-1/2}(S^f)}$ in the above estimates and this shall be crucial in our subsequent argument.

We proceed with the estimate of $\Phi$ in \eqref{eq:mainest1}.
In the sequel, we set
$$
v(\tdx):= u_\delta(A^{-1}(\tdx))= u_\delta(\Bx), \quad \tdx\in D.
$$
Since
$$
\nabla \cdot\Gs_\delta\nabla u_\delta + q_\delta \omega^2 u_\delta =0 \quad \mbox{in} \quad D_\delta\setminus\overline{D}_{\delta/2},
$$
by change of variables, one directly verifies that for $\tdx \in D\setminus \overline{D}_{1/2}$
$$
\gamma (B\nabla_{\tdx})\cdot (B^{-2} B \nabla_{\tdx} u_\delta(A^{-1}(\tdx))+ (\alpha + i\beta) \omega^2 u_\delta(A^{-1}(\tdx))= 0.
$$
That is,
\beq\label{eq:veq1}
\gamma \Delta v + (\alpha +i \beta)\omega^2 v=0 \quad \mbox{in} \quad  D\setminus \overline{D}_{1/2}.
\eeq
In what follows, we shall estimate $\|\tilde\Phi\|_{H^{-3/2}(S^f)}$ and $\|\tilde\Phi\|_{H^{-3/2}(S^a\cup S^b)}$, separately. We recall that the $H^{-3/2}(\p D)$-norm of the function $\tilde\Phi(\tdx)$ is defined as follows
\beq
\|\tilde\Phi\|_{H^{-3/2}(\p D)}= \sup_{\|\varphi\|_{H^{3/2}(\p D)\leq 1}}\Big|\int_{\p D} \tilde\Phi(\tdx)\varphi(\tdx)\, d\Gs_{\tilde{x}}\Big|.
\eeq
Moreover, the $H^{-3/2}(S^c)$-norm of $\tilde\Phi$ for $c\in\{f, a, b\}$ is given as
$$
\|\tilde\Phi\|_{H^{-3/2}(S^c)}:=\sup_{\|\varphi\|_{H^{3/2}_0(S^c)}\leq 1} \Big|\int_{S^c} \tilde\Phi(\tdx)\varphi(\tdx)\, d\Gs_{\tilde{x}}\Big|,
$$
where $H^{3/2}_0(S^c)$ denotes the set of $H^{3/2}(S^c)$-functions which have zero extensions to the whole boundary $\partial D$.
We refer to \cite{Ada,Lio,Wlok87} for more relevant discussions on the Sobolev spaces.

We have
\begin{lem}\label{le:estiphi41}
Let $\tilde\Phi$ be defined in \eqnref{eq:ppp}, where $u_\delta$ is the solution to \eqnref{eq:sys1} with the corresponding $\Gs_\delta$ and $q_\delta$ given by \eqnref{eq:struc1} and \eqnref{eq:loss1}. Then there holds
\beq\label{eq:estphi1}
\|\tilde\Phi\|_{H^{-3/2}(S^f)}\leq C \Big(\gamma + \sqrt{\alpha^2+\beta^2} \omega^2\Big)\|u_\delta\|_{L^2(D_\delta\setminus D_{\delta /2})},
\eeq
and
\beq\label{eq:estphi2}
\|\tilde\Phi\|_{H^{-3/2}(S^a\cup S^b)}\leq C\delta^{-1/2} \Big(\gamma + \sqrt{\alpha^2+\beta^2} \omega^2\Big)
\|u_\delta\|_{L^2(D_\delta\setminus D_{\delta /2})}.
\eeq
\end{lem}

\begin{proof}
For any test function
$\psi\in H^{3/2}_0(S^c)$, $c\in \{f, a, b\}$, we let $\varphi$ be the extension of $\psi$ into $\p D$ such that
\beq\label{eq:estension1}
\varphi=\left\{
\begin{array}{ll}
\psi & \mbox{on} \quad S^c, \\
0 & \mbox{on} \quad \p D\setminus S^c.
\end{array}
\right.
\eeq
We next introduce an auxiliary function $w\in H^2(D)$ such that
\beq\label{eq:testf1}
\left\{
\begin{array}{l}
w= \varphi \quad \mbox{and} \quad \frac{\p w}{\p \nu}=0 \ \ \mbox{on} \ \ \p D,\medskip \\
\|w\|_{H^2(D)} \leq C\|\psi\|_{H^{3/2}(S^c)},\medskip \\
w=0 \quad \mbox{in} \quad D_{1/2}\cup (D\setminus D^c),
\end{array}
\right.
\eeq
The existence of $w$ introduced above can be found in Theorem 14.1, together with its Addendum in \cite{Wlok87}.
From the construction, it is readily seen that $\frac{\p w}{\p \nu}\Big|_{\p D^c}=0$. In the sequel, we let $\varphi_1$ be given as
$$
\varphi_1=\left\{
\begin{array}{ll}
\psi & \mbox{on} \quad S^c, \\
0 & \mbox{on} \quad \p D^c\setminus S^c.
\end{array}
\right.
$$
Then one has
\begin{align*}
&\int_{S^c} \tilde\Phi(\tdx)\psi(\tdx)\, d\Gs_{\tilde{x}}= \int_{S^c} \tilde\Phi(\tdx)\varphi_1(\tdx)\, d\Gs_{\tilde{x}}
+ \delta\gamma\int_{ \p D^c\setminus S^c}\nu_{\tilde{x}}\cdot \nabla_{\tdx} v(\tdx)\Big|_- \varphi_1(\tdx)\, d\Gs_{\tilde{x}} \\
= & \gamma\int_{ S^c}\mathbf\nu_{\tilde{x}}\cdot (B^{-2} B\nabla_{\tdx} v(\tdx)\Big|_- \varphi_1(\tdx)d\Gs_{\tilde{x}} +
\delta\gamma\int_{ \p D^c\setminus S^c} \mathbf\nu_{\tilde{x}} \cdot \nabla_{\tdx}v(\tdx)\Big|_- \varphi_1(\tdx)\, d\Gs_{\tilde{x}}.
\end{align*}
Next by using Green's formula, \eqnref{eq:le22}, \eqnref{eq:testf1} and \eqnref{eq:veq1}, one can derive
\begin{equation}\label{eq:eee1}
\begin{split}
& \Big|\int_{S^c} \tilde\Phi(\tdx)\psi(\tdx)\, d\Gs_{\tilde{x}}\Big| \\
 \leq & \delta\gamma\Big|\int_{\p D^c}\mathbf\nu_{\tilde{x}}\cdot \nabla_{\tdx} v(\tdx)\Big|_- \varphi_1(\tdx)d\Gs_{\tilde{x}} \Big|\\
= & \delta\gamma\Big|\int_{\p D^c} \frac{\p v}{\p \mathbf\nu_{\tilde{x}}}|_- w(\tdx)\, d\Gs_{\tilde{x}}-
\int_{\p D^c} \frac{\p w}{\p \mathbf\nu_{\tilde{x}}}\bigg |_- v(\tdx)\, d\Gs_{\tilde{x}} \Big|\\
= & \delta\gamma \Big|\int_{D^c} \Delta v w- \Delta w v\, d\Gs\Big| \\
\leq & \delta \Big(\sqrt{\alpha^2+\beta^2} \omega^2\|v\|_{L^2(D^c\setminus D^c_{1/2})} \|w\|_{L^2(D^c\setminus D^c_{1/2})}\\
&\qquad\qquad\qquad+ \gamma\|v\|_{L^2(D^c\setminus D^c_{1/2})}\|\Delta w\|_{L^2(D^c\setminus D^c_{1/2})}\Big) \\
\leq & \delta \Big(\gamma + \sqrt{\alpha^2+\beta^2} \omega^2\Big)\|v\|_{L^2(D^c\setminus D^c_{1/2})}\|w\|_{H^{2}(D)}\\
\leq & C\delta \Big(\gamma + \sqrt{\alpha^2+\beta^2}\omega^2\Big)\|v\|_{L^2(D^c\setminus D^c_{1/2})}\|\psi\|_{H^{3/2}(S^c)}.
\end{split}
\end{equation}
By change of variables in integrals, it is straightforward to verify that
\begin{equation}\label{eq:eee2}
\begin{split}
& \|v\|_{L^2(D^f\setminus D^f_{1/2})}=\delta^{-1}\|u_\delta\|_{L^2(D^f_\delta\setminus D_{\delta/2}^f)}, \\
& \|v\|_{L^2(D^c\setminus D^c_{1/2})}=\delta^{-3/2}\|u_\delta\|_{L^2(D^c_\delta\setminus D_{\delta/2}^c)}, \quad c\in \{a, b \}.
\end{split}
\end{equation}
Finally, we have from \eqref{eq:eee1} and \eqref{eq:eee2} that
\begin{equation}\label{eq:eee3}
\Big|\int_{S^f} \tilde\Phi(\tdx)\psi(\tdx)\, d\Gs_{\tilde{x}}\Big|\leq C \Big(\gamma + \sqrt{\alpha^2+\beta^2} \omega^2\Big) \|u_\delta\|_{L^2(D_\delta^f\setminus D^f_{\delta/2})}\|\psi\|_{H^{3/2}(S^f)},
\end{equation}
and for $c\in\{a, b\}$
\begin{equation}\label{eq:eee4}
\Big|\int_{S^c} \tilde\Phi(\tdx)\psi(\tdx)\, d\Gs_{\tilde{x}}\Big|\leq C\delta^{-1/2} \Big(\gamma + \sqrt{\alpha^2+\beta^2} \omega^2\Big) \|u_\delta\|_{L^2(D_\delta^c\setminus D^c_{\delta /2})}\|\psi\|_{H^{3/2}(S^c)}.
\end{equation}
\eqref{eq:eee3} and \eqref{eq:eee4} clearly imply that
$$
\|\tilde\Phi\|_{H^{-3/2}(S^f)}\leq C \Big(\gamma + \sqrt{\alpha^2+\beta^2} \omega^2\Big)\|u_\delta\|_{L^2(D_\delta^f\setminus D^f_{\delta /2})},
$$
and for $c\in\{a, b\}$
$$
\|\tilde\Phi\|_{H^{-3/2}(S^c)}\leq C\delta^{-1/2}\Big(\gamma + \sqrt{\alpha^2+\beta^2} \omega^2\Big)\|u_\delta\|_{L^2(D_\delta^c\setminus D^c_{\delta/2})},
$$
which further imply \eqref{eq:estphi1} and \eqref{eq:estphi2}.

The proof is complete.
\end{proof}

We next derive the estimate of $\|u_\delta\|_{L^2(D_\delta\setminus D_{\delta/2})}$ that is needed in \eqref{eq:estphi1} and \eqref{eq:estphi2} in Lemma~\ref{le:estiphi41}.

\begin{lem}\label{le:48}
Let $u_\delta$ be the solution to \eqnref{eq:sys1} ,
where the corresponding $\Gs_\delta$ and $q_\delta$ are given by \eqnref{eq:struc1}
and \eqnref{eq:loss1}. Then there holds
\beq\label{eq:mainest3}
\|u_\delta\|_{L^2(D_\delta\setminus D_{\delta/2})}^2 \leq  \frac{C}{\beta\omega} |u_\infty|^2 +C \delta \frac{1}{\beta\omega^2}\|\tilde{\Phi}\|_{H^{-3/2}(
S^f)} + C\delta^2 \frac{1}{\beta\omega^2}\|\tilde{\Phi}\|_{H^{-3/2}(\p D)}+ C\delta^2\frac{1}{\beta\omega^2},
\eeq
where $C$ is independent of $\omega$.
\end{lem}

\begin{proof}
By \eqnref{eq:scattering1} and the Sommerfield radiation condition
$$
\frac{\p (u_\delta-u^i)}{\p r}= i\omega (u_\delta-u^i) + \Ocal(|\Bx|^{-2}),
\quad \mbox{as} \quad r=|\Bx|\rightarrow \infty,
$$
it is straightforward to show that
\beq\label{eq:le}
\Big|\Im \lim_{R\rightarrow \infty}\int_{\p B_R}\frac{\p (u_\delta-u^i)}{\p r} (\overline{u}_\delta-\overline{u}^i)\Big|= \omega\lim_{R\rightarrow \infty}\int_{\p B_R}|u_\delta-u^i|^2,
\eeq
where $B_R$ denotes a central ball of radius $R$.
Next, noting that
$$
\Delta u^i +\omega^2 u^i= 0 \quad \mbox{in} \quad \RR^3,
$$
we multiply both sides of the equation \eqnref{eq:sys1}
in $\RR^3\setminus \overline{D}_\delta$ with $\overline{u}_\delta-\overline{u}^i$, and
in $D_\delta$ with $\overline{u}_\delta$, and integrate by parts to have
\begin{align*}
&\lim_{R\rightarrow \infty}\int_{\p B_R}\frac{\p (u_\delta-u^i)}{\p r} (\overline{u}_\delta-\overline{u}^i) -
\int_{\RR^3\setminus \overline{D}_\delta}
|\nabla (u_\delta-u^i)|^2  + \omega^2 \int_{\RR^3\setminus \overline{D}_\delta} |u_\delta-u^i|^2\\
&-\int_{\p D_\delta} \frac{\p (u_\delta-u^i)}{\p \mathbf\nu}\Big|_+ (\overline{u}_\delta-\overline{u}^i)|_+ +
\int_{\p D_\delta} \mathbf\nu\cdot \Gs_\delta \nabla u_\delta\Big|_- \overline{u}_\delta
 -\int_{D_\delta} \nabla\overline{u}_\delta\cdot \Gs_\delta\nabla u_\delta \\
&+ \omega^2\int_{D_\delta} q_\delta|u_\delta|^2=0.
\end{align*}
By taking the imaginary parts of both sides of the last equation and using the definitions of $\Gs_\delta$ and $q_\delta$ in \eqnref{eq:struc1} and \eqnref{eq:loss1}, we arrive at
\begin{align}
\Im \lim_{R\rightarrow \infty}\int_{\p B_R}\frac{\p (u_\delta-u^i)}{\p r} (\overline{u}_\delta-\overline{u}^i)
=&-\beta \omega^2 \int_{D_\delta\setminus \overline{D}_{\delta/2}} |u_\delta|^2 \label{eq:itbyparts1}\\
&-\Im\int_{\p D_\delta} \frac{\p u_\delta}{\p \mathbf\nu}\Big|_+ \overline{u}^i-\Im
\int_{\p D_\delta} \frac{\p u^i}{\p \mathbf\nu} (\overline{u}_\delta-\overline{u}^i)|_+ . \nonumber
\end{align}
We next estimate the last two integral terms in the RHS of \eqref{eq:itbyparts1}. By the definition of $\Phi$ in \eqref{eq:ppp}, one has
\begin{equation}\label{eq:ggg1}
\begin{split}
\Big|\int_{\p D_\delta} \frac{\p u_\delta}{\p \mathbf\nu}\Big|_+ \overline{u}^i\Big|
=& \Big|\int_{\p D_\delta} \Phi\overline{u}^i\Big|\leq C\delta\Big|\int_{S^f} \overline{u}^i(\Bz)
\tilde{\Phi}(\tdy)d\Gs_{\tilde{y}}\Big|+ C \delta^2 \|\tilde{\Phi}\|_{H^{-3/2}(\p D)} \\
\leq & C \delta \|\tilde{\Phi}\|_{H^{-3/2}(S^f)}+ C \delta^2 \|\tilde{\Phi}\|_{H^{-3/2}(\p D)} .
\end{split}
\end{equation}
Similarly, by \eqnref{eq:repre1}, \eqnref{eq:arg2} and \eqnref{eq:estphi11}, one has
\begin{equation}\label{eq:ggg2}
\begin{split}
& \Big|\int_{\p D_\delta} \frac{\p u^i}{\p \mathbf\nu} (\overline{u}_\delta-\overline{u}^i)|_+\Big|=
\Big|\int_{\p D_\delta} \frac{\p u^i}{\p \mathbf\nu}\overline{\Scal_{D_\delta}^\omega [\phi]}\Big|\\
\leq & C\delta^2 \|\tilde\phi\|_{H^{-3/2}(\p D)}\leq C \delta^2 (\|\tilde{\Phi}\|_{H^{-3/2}(\p D)}+1).
\end{split}
\end{equation}
Finally, by combining \eqref{eq:le}, \eqref{eq:itbyparts1}, \eqref{eq:ggg1} and \eqref{eq:ggg2}, we have
\begin{align*}
\beta \omega^2 \int_{D_\delta\setminus \overline{D}_{\delta/2}} |u_\delta|^2\leq
\omega\lim_{R\rightarrow \infty}\int_{\p B_R}|u_\delta-u^i|^2+ C \delta \|\tilde{\Phi}\|_{H^{-3/2}(S^f)}
+ C \delta^2 (\|\tilde{\Phi}\|_{H^{-3/2}(\p D)}+1),
\end{align*}
which together with the fact that
$$
\lim_{R\rightarrow \infty}\int_{\p B_R}|u_\delta-u^i|^2\leq C|u_\infty|^2
$$
readily yields \eqref{eq:mainest3}.

The proof is complete.
\end{proof}

\begin{proof}[Proof of Theorem \ref{th:main1}]
The key ingredient of the proof is to derive the following two estimates
\beq\label{eq:mainest4}
\|\tilde{\Phi}\|_{H^{-3/2}(S^f)}\leq C(\omega)\delta \quad \mbox{and} \quad \|\tilde{\Phi}\|_{H^{-3/2}(\p D)}\leq C(\omega).
\eeq
From \eqnref{eq:estphi2} one can obtain
\begin{equation}\label{eq:ggg3}
\|\tilde{\Phi}\|_{H^{-3/2}(\p D)}^2\leq C (\|\tilde{\Phi}\|_{H^{-3/2}(S^f)}^2 + \|\tilde{\Phi}\|_{H^{-3/2}(S^a\cup S^b)}^2)
\leq C(\omega) \delta^{-1} \|u_\delta\|_{L^2(D_\delta\setminus D_{\delta /2})}^2.
\end{equation}
By \eqnref{eq:scattering1} and \eqnref{eq:mainest1}, we have
\beq\label{eq:mainest5}
|u_\infty(\hat\Bx,\mathbf{d})|\leq C \Big(\delta\|\tilde{\Phi}\|_{H^{-3/2}(S^f)}+ \delta^2(\|\tilde{\Phi}\|_{H^{-3/2}(\p D)}+1)\Big),
\eeq
%
%
and therefore
\begin{equation}\label{eq:ggg4}
|u_\infty(\hat\Bx,\mathbf{d})|^2\leq C\Big(\delta^2\|\tilde{\Phi}\|_{H^{-3/2}(S^f)}^2 +\delta^3\|\tilde{\Phi}\|_{H^{-3/2}(\p D)}^2+\delta^4\Big).
\end{equation}
Then by plugging \eqref{eq:mainest3} and \eqref{eq:ggg4} into \eqref{eq:ggg3}, one can derive
\begin{align*}
\|\tilde{\Phi}\|_{H^{-3/2}(\p D)}^2\leq & C(\omega) \delta^{-1}\Big(
|u_\infty|^2 + \delta \|\tilde{\Phi}\|_{H^{-3/2}(S^f)}+ \delta^2 \|\tilde{\Phi}\|_{H^{-3/2}(\p D)}+\delta^2\Big) \\
\leq &C(\omega)\Big(\delta\|\tilde{\Phi}\|_{H^{-3/2}(S^f)}^2 +\delta^2\|\tilde{\Phi}\|_{H^{-3/2}(\p D)}^2
+\|\tilde{\Phi}\|_{H^{-3/2}(S^f)}+ \delta \|\tilde{\Phi}\|_{H^{-3/2}(\p D)}+\delta\Big) \\
\leq & C(\omega)(\delta \|\tilde{\Phi}\|_{H^{-3/2}(\p D)}^2+ \|\tilde{\Phi}\|_{H^{-3/2}(\p D)}+\delta),
\end{align*}
which readily implies by choosing $\delta\in\mathbb{R}_+$ sufficiently small that
\begin{equation}\label{eq:ggg5}
\|\tilde{\Phi}\|_{H^{-3/2}(\p D)}\leq C(\omega).
\end{equation}
Next by \eqnref{eq:estphi1} and \eqnref{eq:mainest3}, along with the use of \eqref{eq:ggg5},  there holds
\begin{align*}
\|\tilde{\Phi}\|_{H^{-3/2}(S^f)}^2 \leq & C(\omega)\Big(
|u_\infty|^2 + \delta \|\tilde{\Phi}\|_{H^{-3/2}(S^f)}+ \delta^2 \|\tilde{\Phi}\|_{H^{-3/2}(\p D)}+\delta^2\Big)\\
\leq & C(\omega)\Big(\delta^2\|\tilde{\Phi}\|_{H^{-3/2}(S^f)}^2 + \delta\|\tilde{\Phi}\|_{H^{-3/2}(S^f)}
+ \delta^2\|\tilde{\Phi}\|_{H^{-3/2}(\p D)}+\delta^2\Big) \\
\leq &C(\omega)\Big(\delta^2\|\tilde{\Phi}\|_{H^{-3/2}(S^f)}^2 + \delta\|\tilde{\Phi}\|_{H^{-3/2}(S^f)}
+ \delta^2\Big),
\end{align*}
which in turn implies by taking $\delta\in\mathbb{R}_+$ sufficiently small that
\begin{equation}\label{eq:ggg6}
\|\tilde{\Phi}\|_{H^{-3/2}(S^f)}\leq C(\omega)\delta.
\end{equation}
Finally, by inserting \eqnref{eq:ggg5} and \eqref{eq:ggg6} into \eqnref{eq:mainest5}, we immediately have \eqref{eq:ncest1}.

The proof is complete.
\end{proof}

\begin{rem}\label{rem:general}
For our study on the regularized full-cloak in the present section, we have made use of $D_\delta=D_\delta^f\cup D_\delta^a\cup D_\delta^b$ as the virtual domain for the blowup construction. However, we would like to remark that one can also simply use $D_\delta^f$ as the virtual domain by excluding the two ``caps", $D_\delta^a$ and $D_\delta^b$. In such a case, one could replace $S_\delta^a$ and $S_\delta^b$, respectively, by $\mathscr{S}_\delta(P_0)$ and $\mathscr{S}_\delta(Q_0)$ in our earlier arguments. Then by following a similar argument, one can arrive at the same conclusion as earlier for the approximate full-cloak. The reason that we have chosen to work with $D_\delta$ instead of $D_\delta^f$ as the virtual domain are two-folded. First, we would like to include the geometry considered in \cite{LiLiuRonUhl} for the regularized full-cloaks as a particular case in our this study; see also the discussion made after \eqref{eq:jacob1}. Second, the corresponding mathematical arguments of dealing with $D_\delta$ are more general than those of dealing with $D_\delta^f$, and we appeal to presenting a more general mathematical study. 
\end{rem}


\section{Regularized partial-cloak}\label{sect:5}
In this section, we consider the regularized partial-cloak by taking the generating set $\GG_0$ to be an open subset on a flat plane $\mathbb{P}_0$ in $\mathbb{R}^3$. Without loss of generality, we assume that $\mathbb{P}_0$ is the $\{x_3=0\}$-plane. If $\GG_0$ is regarded as a domain in $\mathbb{R}^2$, it is assumed that $\GG_0$ is bounded and simply connected with a convex boundary. However, in order to ease the exposition and illustration, we shall confine ourselves to a special case by taking $\Gamma_0$ to be a square throughout the rest of the paper, which was actually considered in \cite{LiLiuRonUhl}. Nevertheless, at this point, we would like to emphasize that our subsequent study on the partial cloaking can be easily adapted to deal with a more general generating set $\Gamma_0$.  

Let $\mathbf{n}\in\mathbb{S}^2$ be the unit normal vector to $\GG_0$ and let $q\in\mathbb{R}_+$. We next introduce the virtual domain $D_\delta$ for the blowup construction of the regularized partial-cloak; see Fig.~\ref{fig2} for a schematic illustration.
\begin{figure}[h]
\begin{center}
  \includegraphics[width=3.5in,height=2.0in]{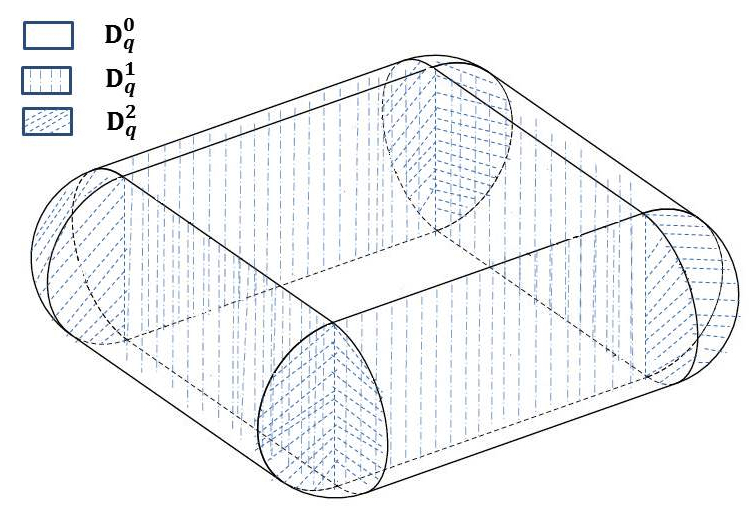}
  \end{center}
  \caption{Schematic illustration of the domain $D_q$ for the regularized partial-cloak. \label{fig2}}
\end{figure}
Let
\begin{equation}\label{eq:rrr1}
D_q^0:=\GG_0(\Bx)\times [\Bx-\tau\cdot \mathbf{n}, \Bx+\tau\cdot \mathbf{n}],\quad \Bx\in\overline{\GG}_0,\quad 0\leq \tau\leq q,
\end{equation}
where we identify $\GG_0$ with its parametric representation $\GG_0(\Bx)$. We denote by $D_q^{1}$ the union of the four side
half-cylinders and $D_q^{2}$ the union of four corner quarter-balls in Fig.~\ref{fig2}. Moreover, we set $S_q^1$ and $S_q^2$ to be
$$
S_q^1:=\p D_q \cap \p D_q^1, \quad S_q^2:=\p D_q \cap \p D_q^2.
$$
The upper and lower surfaces of $D_q$ in Fig.~\ref{fig2} are, respectively, denoted by
$$
\GG_q^1:=\{\Bx+ q\cdot \mathbf{n};\, \Bx\in \GG_0\} , \quad \GG_q^2:=\{\Bx-q\cdot \mathbf{n};\, \Bx\in \GG_0\}.
$$
Define $S_q^0:=\GG_q^1\cup \GG_q^2$.
We then have 
\begin{equation}\label{eq:vvnn1}
D_q=D_q^0\cup D_q^1\cup D_q^2\qquad\mbox{and}\qquad \partial D_q=S_q^0\cup S_q^1\cup S_q^2.
\end{equation} 
Let $\delta\in\mathbb{R}_+$ be the asymptotically small regularization parameter. We let $D_\delta$ be the virtual domain and its boundary is clearly given by
\beq\label{eq:struc51}
\p D_\delta= S_\delta^0 \cup S_\delta^1 \cup S_\delta^2.
\eeq
In what follows, $D_\delta$ shall be referred to as a {\it screen-like} region. If $q\equiv 1$, we shall drop the dependence on $q$ of $D_q, S_q^0$,
$S_q^1$ and $S_q^2$, and simply write them as $D, S^0$, $S^1$, and $S^2$.

We shall make use of the blowup transformation $A$ introduced in \cite{LiLiuRonUhl} that maps $D_\delta$ to $D$, and refer to that work for the detailed construction. For our subsequent use, it is stressed that in $D_\delta^0$ the blow-up transformation takes the following form
$$
A(\By)=\tdy:=\left(\frac{\mathbf{e}_3\mathbf{e}_3^T}{\delta}+
\mathbf{e}_1\mathbf{e}_1^T+\mathbf{e}_2\mathbf{e}_2^T\right)\By, \quad \By\in D_\delta^0,
$$
where $\By\in D_\delta$ and $\tdy\in D$, with the three Euclidean unit vectors given as follows
$$
\mathbf{e}_1=(1, 0, 0)^T,\quad \mathbf{e}_2=(0, 1, 0)^T, \quad \mathbf{e}_3=(0, 0, 1)^T.
$$
Now we introduce the lossy layer for our partial-cloaking device as follows
\beq\label{eq:loss2}
\Gs_\delta(\Bx)=\Gs_l(\Bx):=\gamma B^{-2}(\Bx), \quad q_\delta(\Bx)=q_l(\Bx): = \alpha + i \beta \quad \mbox{for} \quad \Bx\in D_\delta\setminus\overline{D_{\delta/2}},
\eeq
where $B(\Bx):=\nabla_x A(\Bx)$ is the Jacobian matrix of the blowup transformation $A$.

The following theorem quantifies our partial-cloaking construction.

\begin{thm}\label{th:main2}
Let $D_\delta$ be a screen-like region described above with its boundary given in \eqnref{eq:struc51}.
Let $u_\delta$ be the solution to \eqnref{eq:sys1} corresponding to $\Gs_\delta$ and $q_\delta$ defined in \eqnref{eq:struc1}, with $(D_\delta\backslash\overline{D_{\delta/2}}; \sigma_l, q_l)$ given by \eqnref{eq:loss2}, and $(D_{\delta/2}; \Gs_a, q_a)$ being arbitrary but regular.
Let $u^i(\Bx)=e^{i\omega\Bx\cdot \Bd}$ be the incident plane wave satisfying
\beq
|\mathbf{d}\cdot \mathbf{n}| \leq \epsilon , \quad \epsilon \ll 1,
\eeq
and let $u_\infty^\delta(\hat\Bx,\mathbf{d})$ be the scattering amplitude of $u_\delta$.
Then there exists $\delta_0\in\mathbb{R}_+$ such that when $\delta<\delta_0$,
 $u^\delta_\infty(\hat{\Bx},\mathbf{d})$ satisfies
\beq\label{eq:88}
|u^\delta_\infty(\hat{\Bx},\mathbf{d})|\leq C(\omega) (\epsilon + \delta),
\eeq
where $C(\omega)$ is a positive constant depending on $\omega$ and $D$, but independent of $\Gs_a$, $q_a$ and $\hat\Bx$, $\mathbf{d}$.
\end{thm}

\begin{rem}\label{lem:physical12}
Similar to Remark~\ref{lem:physical1}, one can immediately infer by Theorem~\ref{th:main2} that the push-forwarded structure
$$\{\Omega_e; \Gs, q\}=(F_\delta)_*\{\Omega; \Gs_\delta, q_\delta\}$$
with $F_\delta$ defined by \eqnref{eq:blow1} and $\{\Omega_e; \Gs_\delta, q_\delta\}$ defined in the theorem, produces an approximate partial cloaking device. The incident aperture of the approximate partial-cloak is given by
\[
\Lambda_i=\{\Bd\in\mathbb{S}^2;\ |\Bd\cdot \Bn|\leq \epsilon\},
\]
whereas the observation aperture is $\Lambda_o=\mathbb{S}^2$.
\end{rem}

\subsection{Asymptotic expansions}
Similar to our notational usage in Section~\ref{sect:4}, if we let $\Bz$ denote the space variable on $\Gamma_0$, then for any $\By\in \p D_\delta$, we define
$\mathbf{z}_{y}$ to be the projection of $\By$ onto $\GG_0$. Denote by $\nabla_n$ the normal derivative with respect to $\GG_0$.
Then for $\Bx\in \RR^3\setminus \overline{D}_\delta$ we have
\begin{align}
\Scal_{D_\delta}^\omega[\phi](\Bx)= & \int_{\p D_\delta} G_\omega(\Bx-\By) \phi(\By) d\Gs_y \nonumber\\
= & \int_{S^0} G_\omega(\Bx-\mathbf{z}_{\tilde{y}})\tilde\phi(\tdy)d\Gs_{\tilde{y}}+\delta \int_{S^1} G_\omega(\Bx-\mathbf{z}_{\tilde{y}}) \tilde\phi(\tdy)d\Gs_{\tilde{y}} \label{eq:expan51}\\
& + \delta \int_{S^0} (\tdy-\mathbf{z}_{\tilde{y}})\cdot\nabla_{n_y} G_\omega(\Bx-\mathbf{z}_{\tilde{y}})\tilde\phi(\tdy)d\Gs_{\tilde{y}}+ \Ocal\Big(\delta^2|\Bx|^{-1}\|\tilde\phi\|_{H^{-3/2}(\p D)}\Big),\nonumber
\end{align}
where and also in what follows, $\tilde\phi(\tdy):=\phi(\By)$ and $\tdy:= A(\By)\in \overline{D}$ for $\By\in \overline{D_\delta}$.

We next estimate $\tilde\phi(\tdy)$ in \eqref{eq:expan51}, and first derive the following lemma.
\begin{lem}\label{le:5.1}
Let $D_\delta$ be a screen-like region described earlier. Then for $\Bx\in \p D_\delta$ we have
\beq\label{eq:arg3}
\begin{split}
(\Kcal_{D_\delta}^\omega)^*[\phi](\Bx)=&
\left\{
\begin{array}{ll}
 (\Kcal_{S^0}^\omega)^*[\tilde\phi](\tdx)+ \Ocal(\delta\|\tilde\phi\|_{H^{-3/2}(\p D)}), & \tdx \in S^0\cup S^1, \\
 (\Kcal_{S^0}^\omega)^*[\tilde\phi](\tdx)+(\Kcal_{S^2})^*[\tilde\phi](\tdx)+ \Ocal(\delta\|\tilde\phi\|_{H^{-3/2}(\p D)}), & \tdx \in S^2,
 \end{array}
 \right.\\
\Scal_{D_\delta}^\omega[\phi](\Bx)=&
\left\{
\begin{array}{ll}
\Scal_{S^0}^\omega[\tilde\phi](\tdx)+ \Ocal(\delta\|\tilde\phi\|_{H^{-3/2}(\p D)}), & \tdx \in S^0\cup S^1, \\
\Scal_{S^0}^\omega[\tilde\phi](\tdx)+\Scal_{S^2}[\tilde\phi](\tdx)+ \Ocal(\delta\|\tilde\phi\|_{H^{-3/2}(\p D)}), & \tdx \in S^2,
 \end{array}
 \right.
\end{split}
\eeq
where $ (\Kcal_{S^0}^\omega)^*[\tilde\phi](\tdx)$ is defined by
$$
(\Kcal_{S^0}^\omega)^*[\tilde\phi](\tdx)=\frac{1}{4\pi}\int_{S^0}\Big(\frac{\la \mathbf{z}_{\tilde{x}}-\mathbf{z}_{\tilde{y}},\mathbf\nu_x\ra}{|\mathbf{z}_{\tilde{x}}-\mathbf{z}_{\tilde{y}}|^3}
-i\omega\frac{\la \mathbf{z}_{\tilde{x}}-\mathbf{z}_{\tilde{y}},\mathbf\nu_x\ra}{|\mathbf{z}_{\tilde{x}}-\mathbf{z}_{\tilde{y}}|^2}\Big)
e^{i\omega|\mathbf{z}_{\tilde{x}}-\mathbf{z}_{\tilde{y}}|}\tilde\phi(\tdy)d\Gs_{\tilde{y}},
$$
and
$$
\Scal_{S^0}^\omega[\tilde\phi](\tdx)=\int_{S^0}G_\omega{(\mathbf{z}_{\tilde{x}}-\mathbf{z}_{\tilde{y}})}
\tilde\phi(\tdy)d\Gs_{\tilde{y}}.
$$
\end{lem}

\begin{proof}
By following a similar argument to that in the proof of Lemma \ref{le:1}, we compute for $\Bx\in S_\delta^0\cup S_\delta^1$
$$
|\Bx-\By|=|\Bx-\mathbf{z}_x-(\By-\mathbf{z}_y)+ \mathbf{z}_x-\mathbf{z}_y|=|\mathbf{z}_{\tilde{x}}-\mathbf{z}_{\tilde{y}}|+ \Ocal(\delta),
$$
and
$$
\la \Bx-\By, \mathbf{\nu}_{x}\ra= \la\mathbf{z}_x-\mathbf{z}_y, \mathbf{\nu}_{\tilde{x}}\ra+\Ocal(\delta),
\quad e^{i\omega|\Bx-\By|}= e^{i\omega|\mathbf{z}_{\tilde{x}}-\mathbf{z}_{\tilde{y}}|}+\Ocal(\delta).
$$
Note that for $\Bx\in S_\delta^2$ one has
$$
|\Bx-\By|=\delta|\tdx-\tdy|, \quad e^{i\omega|\Bx-\By|}=1+ \Ocal(\delta).
$$
The proof can then be completed by using the similar expansion method as that in the proof of Lemma \ref{le:1}.
\end{proof}

Next we introduce an auxiliary scattering problem for our partial-cloaking study.  Let $v_\delta$ be the solution to the following scattering system
\beq\label{eq:sndhd1}
\left\{
\begin{array}{l}
\Delta v_\delta + \omega^2 v_\delta =0  \quad \mbox{in} \quad \RR^3\setminus \overline{D_\delta}\medskip\\
\displaystyle{\frac{\p v_\delta}{\p \mathbf\nu}=0}  \ \qquad\quad\quad\mbox{on} \quad \p D_\delta\medskip\\
v^+:=v_\delta- u^i \quad \mbox{satisfies the radiation condition } \eqnref{eq:rad}
\end{array}
\right.
\eeq
Then we have
\begin{thm}
Let $u_\delta$ be the solution to \eqnref{eq:sndhd1}. Then there holds
the following far-field expansion
\beq\label{eq:expan53}
(v_\delta-u^i)(\Bx)= \int_{\GG_0}G_\omega(\Bx-\mathbf{z}_{\tilde{y}})\mathrm{K}[\mathbf{n}\cdot\nabla u^i(\mathbf{z})](\Bz_{\tilde y})\, d\Gs_{z_{\tilde{y}}}+\Ocal(\delta),
\eeq
where
\beq\label{eq:op511}
\mathrm{K}:=\left(\frac{I}{4}-(\Kcal_{\GG_0}^\omega)^*\right)^{-1}(\Kcal_{\GG_0}^\omega)^*
\left(\frac{I}{4}+(\Kcal_{\GG_0}^\omega)^*\right)^{-1}
\eeq
with the operator $(\Kcal_{\GG_0}^{\omega})^*$ defined by
\beq\label{eq:op512}
(\Kcal_{\GG_0}^{\omega})^*[\varphi]:=\rm{p.v. } \int_{\GG_0} \mathbf{n}\cdot\nabla_z G_\omega(\Bz-\By) \phi(\By)\, d\Gs_z  .
\eeq
\end{thm}

\begin{proof}

The solution to \eqnref{eq:sndhd1} can be represented by \eqnref{eq:repre1},
where the potential $\phi$ satisfies
\beq\label{eq:trans2}
\Big(\frac{I}{2} + (\Kcal_{D_\delta}^\omega)^*\Big)[\phi]=-\frac{\p u^i}{\p \mathbf\nu} \quad \mbox{on} \quad  \p D_\delta.
\eeq
Using the fact that
$$
\frac{I}{2}+(\Kcal_{S^0}^\omega)^* : H^{-3/2}(S^0)\rightarrow H^{-3/2}(S^0)
$$
is invertible (cf. \cite{nedelec,Tay}), and by \eqnref{eq:trans2} one has
\begin{equation}\label{eq:111}
\tilde{\phi}(\tdy)= -\Big(\frac{I}{2}+(\Kcal_{S^0}^\omega)^*\Big)^{-1}\left[\frac{\p u^i}{\p \mathbf\nu}\right](\tdy) + \Ocal(\delta), \quad
\tdy \in \p D.
\end{equation}
Next we can expand $u^i$ with respect to $\mathbf{z}$ as follows
\begin{equation}\label{eq:222}
u^i(\By)= u^i(\mathbf{z}_y)+ \delta \nabla u^i(\mathbf{z}_y)\cdot (\tdy-\mathbf{z}_y) + \Ocal(\delta^2), \quad \By\in \p D_\delta.
\end{equation}
By \eqref{eq:111} and \eqref{eq:222}, one further has
\beq\label{eq:expan52}
\tilde{\phi}(\tdy)= -\Big(\frac{I}{2}+(\Kcal_{S^0}^\omega)^*\Big)^{-1}\Big[\mathbf\nu_{\tilde{x}}\cdot\nabla u^i(\mathbf{z}_{\tilde{x}})\Big](\tdy) + \Ocal(\delta), \quad \tdy\in \GG^1,
\eeq
and hence $\|\tilde\phi\|_{H^{-3/2}(\p D)}\leq C(\omega)$.
Noting that for $\tdy\in \GG^1$, $(\tdy-2\mathbf{n})\in \GG^2$, we define
$\tilde\psi(\tdy):=\tilde\phi(\tdy-2\mathbf{n})$ for $\tdy\in \GG^1$ and
$\tilde\psi(\tdy):=\tilde\phi (\tdy+2\mathbf{n})$ for $\tdy\in \GG^2$. By using the fact that
$$
\mathbf{\nu}_{\tilde{y}-2n_{\tilde{y}}}= - \mathbf{\nu}_{\tilde{y}}=-\mathbf{n} , \quad \mbox{for} \quad \tdy\in \GG^1,
$$
one obtains (assume for a while that $\tdx\in S^0$)
\beq\label{eq:expan522}
\tilde{\phi}(\tdy-2\mathbf{n})=\tilde\psi(\tdy)=\Big(\frac{I}{2}-(\Kcal_{S^0}^\omega)^*\Big)^{-1}\Big[\mathbf{n}\cdot\nabla u^i(\mathbf{z}_{\tilde{x}})\Big](\tdy) + \Ocal(\delta), \quad \tdy\in \GG^1.
\eeq
By inserting \eqnref{eq:expan52} and \eqnref{eq:expan522} into \eqnref{eq:expan51}, there holds for $\Bx\in \RR^3\setminus \overline{D}$
\begin{align*}
&\Scal_{D_\delta}^\omega[\phi](\Bx)=\int_{S^0} G_\omega(\Bx-\mathbf{z}_{\tilde{y}})\tilde\phi(\tdy)d\Gs_{\tilde{y}}+ \Ocal(\delta) \\
=& \int_{\GG^1} G_\omega(\Bx-\mathbf{z}_{\tilde{y}})\tilde\phi(\tdy)d\Gs_{\tilde{y}}+
\int_{\GG^1} G_\omega(\Bx-\mathbf{z}_{\tilde{y}})\tilde\psi(\tdy)d\Gs_{\tilde{y}}+ \Ocal(\delta)\\
=& \int_{\GG_0}G_\omega(\Bx-\mathbf{z}_{\tilde{y}})\left(\Big(\frac{I}{2}-(\Kcal_{S^0}^\omega)^*\Big)^{-1}
-\Big(\frac{I}{2}+(\Kcal_{S^0}^\omega)^*\Big)^{-1}
\right)\Big[\p_n u^i(\mathbf{z})\Big](\Bz_{\tilde{y}})d\Gs_{z_{\tilde{y}}}\\
& +\Ocal(\delta)=\int_{\GG_0}G_\omega(\Bx-\mathbf{z}_{\tilde{y}})\hat{\mathrm{K}}[\p_n u^i(\mathbf{z})](\Bz_{\tilde{y}})d\Gs_{z_{\tilde{y}}}+\Ocal(\delta),
\end{align*}
where $\p_n u^i(\mathbf{z})=\mathbf{n}\cdot\nabla u^i(\mathbf{z})$ and $\hat{\mathrm{K}}$ is defined by
\beq\label{eq:op51}
\hat{\mathrm{K}}:=2\left(\frac{I}{2}-(\Kcal_{S^0}^\omega)^*\right)^{-1}(\Kcal_{S^0}^\omega)^*
\left(\frac{I}{2}+(\Kcal_{S^0}^\omega)^*\right)^{-1}.
\eeq
It is noted in \eqref{eq:expan53} that the function $\Theta(\Bz):=\mathbf{n}\cdot\nabla u^i(\mathbf{z})$ depends only on $\Bz\in \GG_0$, one can readily verify that
$$
(\Kcal_{\GG_0}^\omega)^*[\Theta](\Bz)=\frac{1}{2}(\Kcal_{S^0}^\omega)^*[\hat\Theta(\tdy)](\Bz+\mathbf{n}), \quad \Bz\in \GG_0,
$$
for $\hat\Theta(\tdy)=\mathbf{n}\cdot\nabla u^i(\mathbf{z}_{\tilde{y}})$, $\tdy\in \GG_1\cup \GG_2$.
 Hence, we can replace the operator $\hat{\mathrm{K}}$ in \eqnref{eq:op51} to be
$$
\mathrm{K}:=\left(\frac{I}{4}-(\Kcal_{\GG_0}^\omega)^*\right)^{-1}(\Kcal_{\GG_0}^\omega)^*
\left(\frac{I}{4}+(\Kcal_{\GG_0}^\omega)^*\right)^{-1}.
$$
The proof is complete.
\end{proof}

\begin{thm}\label{th:pc1}
Let $v_\delta$ be the solution to \eqnref{eq:sndhd1} and
$v^\delta_\infty(\hat{\Bx}, \mathbf{d})$ be its scattering amplitude. Then there holds
\beq\label{eq:pcscatamp1}
\Big|v^\delta_\infty(\hat{\Bx}, \mathbf{d})+\frac{1}{2\pi}\int_{\Gamma_0}e^{-i\omega\frac{4\pi}{3}\sum\limits_{m=-1}^1 Y_1^m
(\hat{\Bx})\overline{Y_1^m(\hat{\Bz})}|\Bz|}\mathrm{K}[\mathbf{n}\cdot\nabla u^i](\Bz)\ d\sigma_z\Big|\leq C(\omega)\delta,
\eeq
where $\mathrm{K}$ is defined in \eqref{eq:op511}.
\end{thm}

\begin{proof}
We recall the following addition formula (see, e.g., \cite{nedelec})
\begin{align}
\frac{1}{|\Bx-\By|}= 4\pi\sum_{n=0}^\infty \sum_{m=-n}^{n} \frac{1}{2n+1}Y_n^m(\xi, \vartheta) \overline{Y_n^m(\xi', \vartheta')} \,\frac{r'^n}{r^{n+1}},\label{Gammaexp1}
\end{align}
where $(r,\xi, \vartheta)$ and $(r',\xi', \vartheta')$ are the
spherical coordinates of $\Bx$ and $\By$, respectively; and $Y_n^m$ is the spherical harmonic function of degree $n$ and order $m$.
For simplicity, the parameters $(r, \xi, \vartheta)$ and $(r, \xi', \vartheta')$ shall be replaced by $(|\Bx|, \hat{\Bx})$
and $(|\By|, \hat{\By})$, respectively. It then follows from \eqnref{Gammaexp1} that
\begin{equation}\label{eq:333}
|\Bx-\By|=|\Bx|\Big(1-\frac{4\pi}{3}\sum_{m=-1}^{1}Y_1^m(\hat{\Bx})\overline{Y_1^m(\hat{\By})}\frac{|\By|}{|\Bx|}\Big)
+\Ocal(|\Bx|^{-1}).
\end{equation}
By \eqref{eq:333}, we have the following expansion
\begin{equation}\label{eq:444}
\begin{split}
& (v_\delta-u^i)(\Bx)=\int_{\p D_\delta}G_\omega(\Bx-\By) \phi(\By)d\Gs_y\\
= & -\frac{1}{4\pi}\frac{e^{i\omega|\Bx|}}{|\Bx|}\int_{\p D_\delta}e^{-i\omega\frac{4\pi}{3}\sum\limits_{m=-1}^1 Y_1^m
(\hat{\Bx})\overline{Y_1^m(\hat{\By})}|\By|}\phi(\By)\, d\Gs_y +\Ocal(|\Bx|^{-2}).
\end{split}
\end{equation}
Expanding the integral kernel function in \eqref{eq:444} around $\Bz_y$, one further has
\begin{align*}
v^\delta_\infty(\hat{\Bx},\mathbf{d})= &-\frac{1}{4\pi}\int_{\p D_\delta}e^{-i\omega\frac{4\pi}{3}\sum\limits_{m=-1}^1 Y_1^m
(\hat{\Bx})\overline{Y_1^m(\hat{\By})}|\By|}\phi(\By)\, d\Gs_y \\
= & -\frac{1}{4\pi}\int_{\Gamma_1\cup \Gamma_2}e^{-i\omega\frac{4\pi}{3}\sum\limits_{m=-1}^1 Y_1^m
(\hat{\Bx})\overline{Y_1^m(\hat{\Bz}_{\tilde{y}})}|\Bz_{\tilde{y}}|}\tilde{\phi}(\tdy)\, d\Gs_{\tilde{y}}+\Ocal(\delta),
\end{align*}
which together with \eqref{eq:expan52} readily yields
\begin{align*}
v^\delta_\infty(\hat{\Bx},\mathbf{d})=-\frac{1}{2\pi}\int_{\Gamma_0}e^{-i\omega\frac{4\pi}{3}\sum\limits_{m=-1}^1 Y_1^m
(\hat{\Bx})\overline{Y_1^m(\hat{\Bz})}|\Bz|}\mathrm{K}[\mathbf{n}\cdot\nabla u^i](\Bz)\, d\sigma_z+\Ocal(\delta).
\end{align*}

The proof is complete.
\end{proof}

{
In view of Theorem \ref{th:pc1}, one immediately has the following result
\begin{prop}\label{prop:pc1}
Let $v_\delta$ be the solution to \eqnref{eq:sndhd1} and $v^\delta_\infty(\hat\Bx,\Bd)$ be its scattering amplitude. Suppose that the impinging direction $\Bd$ of the corresponding incident plane wave $u^i$ satisfies
\beq
|\mathbf{d}\cdot \mathbf{n}| \leq \epsilon , \quad \epsilon \ll 1.
\eeq
Then for sufficient small $\delta\in\mathbb{R}_+$ there holds
\beq
|v_\infty(\hat{\Bx},\mathbf{d})|\leq C(\omega) (\epsilon + \delta),
\eeq
where $C(\omega)$ depends only on $\GG_0$ and $\omega$.
\end{prop}
\begin{proof}
Let $\mathrm{K}$ be defined in \eqnref{eq:op511}. Since
$$
\|\mathrm{K}[\mathbf{n}\cdot\nabla u^i(\mathbf{z})]\|_{H^{-1/2}(\GG_0)}\leq C(\omega) \|\mathbf{n}\cdot\nabla u^i(\mathbf{z})\|_{H^{-1/2}(\GG_0)},
$$
which together with \eqnref{eq:pcscatamp1} implies that
$$
|v_\infty(\hat{\Bx},\mathbf{d})|\leq C(\omega)(\|\mathbf{n}\cdot\nabla u^i(\mathbf{z})\|_{H^{-1/2}(\GG_0)}+\delta).
$$
The proof can be completed by noting that
$$
\mathbf{n}\cdot\nabla u^i(\mathbf{z})=i\omega \mathbf{n}\cdot \mathbf{d} e^{i\omega \mathbf{z}\cdot \mathbf{d}}.
$$
\end{proof}

\subsection{Partial cloak}

\begin{proof}[Proof of Theorem \ref{th:main2}]
We follow similar arguments to those in the proof of Theorem \ref{th:main1}. Define $\tilde\phi(\tdy):=\phi(\By)$.
Then it follows from
$$
\Big(\frac{I}{2} + (\Kcal_{D_\delta}^\omega)^*\Big)[\phi]=\frac{\p u_\delta}{\p \mathbf\nu}\Big|_+-\frac{\p u^i}{\p \mathbf\nu} \quad \mbox{on} \quad  \p D_\delta,
$$
and Lemma \ref{le:5.1}, and the proof of Proposition \ref{prop:pc1} that
\beq\label{eq:mainpftmp51}
\|\tilde\phi\|_{H^{-3/2}(S^j)} \leq C(\|\tilde\Phi\|_{H^{-3/2}(S^j)}+\|\nu\cdot \nabla u(\Bz)\|_{H^{-3/2}(\p D)}+\delta), \quad j=0, 1, 2.
\eeq
By \eqnref{eq:expan51}, there holds
$$
|u_\delta-u^i| = |\Scal_{D_\delta}[\phi]|=\left|\int_{S^0} G_\omega(\Bx-\mathbf{z}_{\tilde{y}})
\tilde\phi(\tdy)d\Gs_{\tilde{y}}\right|+\Ocal(\delta(\|\tilde\phi\|_{H^{-3/2}(S^0\cup S^1)}+\delta\|\tilde\phi\|_{H^{-3/2}(\p D)})),
$$
which together with \eqnref{eq:mainpftmp51} further implies
\beq\label{eq:est51}
|u_\infty(\hat{\Bx},\mathbf{d})|\leq C(\omega)\Big(\|\tilde{\Phi}\|_{H^{-3/2}(S^0)}+\epsilon +\delta + \delta\|\tilde{\Phi}\|_{H^{-3/2}(S^1)}
+\delta^2\|\tilde{\Phi}\|_{H^{-3/2}(S^2)}\Big),
\eeq
where $\tilde\Phi(\tdx)=\Phi(\Bx):=\frac{\p u_\delta}{\p \mathbf\nu}\Big|_+(\Bx)$ for $\Bx\in \partial D_\delta$.
The following estimates can be obtained by using a completely similar argument as that in the proof of Lemma \ref{le:estiphi41},
\beq\label{eq:tmp51}
\begin{split}
\|\tilde{\Phi}\|_{H^{-3/2}(S^j)}\leq& C(\omega)\delta^{-(j-1)/2}\|u_\delta\|_{L^2(D_\delta\setminus D_{\delta/2})}, \quad j=0, 1, 2.
\end{split}
\eeq
Inserting \eqnref{eq:tmp51} back into \eqnref{eq:est51}, we have
\beq\label{eq:farfieldest1}
|u_\infty(\hat{\Bx},\mathbf{d})|\leq C(\omega)\Big(\delta^{1/2}\|u_\delta\|_{L^2(D_\delta\setminus D_{\delta/2})}+\epsilon +\delta\Big).
\eeq
Following a similar argument to that in the proof of Lemma \ref{le:48}, one can show \eqnref{eq:itbyparts1}.
By change of variables in integrals, one further has
\begin{equation}\label{eq:44}
\begin{split}
\Big|\int_{\p D_\delta} \frac{\p u_\delta}{\p \mathbf\nu}\Big|_+ \overline{u}^i\Big|
\leq & \Big|\int_{S^0} \overline{u}^i(\Bz)
\tilde{\Phi}(\tdy)d\Gs_{\tilde{y}}\Big|+ C(\omega)\delta\|\tilde\Phi\|_{H^{-3/2}(\p D)} \\
\leq & C(\omega) \|\tilde{\Phi}\|_{H^{-3/2}(S^0)}+
C(\omega)\delta\|\tilde\Phi\|_{H^{-3/2}(\p D)},
\end{split}
\end{equation}
which together with \eqnref{eq:arg3} readily yields
\begin{equation}\label{eq:55}
\begin{split}
&\Big|\int_{\p D_\delta} \frac{\p u^i}{\p \mathbf\nu} (\overline{u}_\delta-\overline{u}^i)|_+\Big|=
\Big|\int_{\p D_\delta} \frac{\p u^i}{\p \mathbf\nu}\overline{\Scal_{D_\delta}^\omega [\phi]}\Big|\\
&\leq \Big|\int_{S^0} \frac{\p u^i}{\p \mathbf\nu}(\Bz)\Scal_{S^0}^\omega [\tilde\phi]\Big|+
C\delta(\|\tilde\Phi\|_{H^{-3/2}(\p D)}+1) \\
&\leq C(\omega) \epsilon +C\delta(\|\tilde\Phi\|_{H^{-3/2}(\p D)}+1).
\end{split}
\end{equation}
By taking the imaginary parts of both sides of \eqnref{eq:itbyparts1}, and using \eqnref{eq:tmp51}, \eqnref{eq:farfieldest1}, \eqref{eq:44} and \eqref{eq:55}, there holds
\begin{equation}\label{eq:66}
\begin{split}
&\|u_\delta\|_{L^2(D_\delta\setminus D_{\delta/2})}^2\\
 \leq & C(\omega)|u_\infty|^2 + C(\omega) \|\tilde{\Phi}\|_{H^{-3/2}(S^0)}
+C(\omega) \epsilon+C(\omega)\delta(\|\tilde\Phi\|_{H^{-3/2}(\p D)}+1)\\
\leq & C(\omega) \Big(\delta\|u_\delta\|_{L^2(D_\delta\setminus D_{\delta/2})}^2+ \delta^{1/2}\|u_\delta\|_{L^2(D_\delta\setminus D_{\delta/2})}
+ \epsilon+\delta\Big).
\end{split}
\end{equation}
Then by choosing a sufficiently small $\delta\in\mathbb{R}_+$, one easily has from \eqref{eq:66} that
\begin{equation}\label{eq:77}
\|u_\delta\|_{L^2(D_\delta\setminus D_{\delta/2})}\leq C(\omega)( \epsilon^{1/2}+\delta^{1/2}).
\end{equation}
Plugging \eqref{eq:77} into \eqnref{eq:tmp51}, one finally has
$$
\|\tilde{\Phi}\|_{H^{-3/2}(S^0)}\leq C(\omega)(\epsilon+ \delta)\quad
\mbox{and} \quad \|\tilde{\Phi}\|_{H^{-3/2}(\p D)}\leq C(\omega)(\epsilon^{1/2} + \delta^{1/2}),
$$
which together with \eqnref{eq:est51} readily yields \eqref{eq:88}.

The proof is complete.
\end{proof}

\section*{Acknowledgment}

The work of Y. Deng was partially supported by the Mathematics and Interdisciplinary Sciences Project, Central South University, China. The work of H. Liu was partially supported by Hong Kong RGC General Research Funds, HKBU 12302415 and 405513, and the NSF grant of China, No. 11371115. The work of G. Uhlmann was supported by NSF.

\end{document}